\documentclass{article}%
\usepackage{graphicx}
\usepackage{amsmath}
\usepackage{amsfonts}
\usepackage{amssymb}%
\newtheorem{theorem}{Theorem}[section]
\newtheorem{corollary}[theorem]{Corollary}
\newtheorem{lemma}[theorem]{Lemma}
\newtheorem{proposition}[theorem]{Proposition}
\newenvironment{proof}[1][Proof]{\textbf{#1.} }{\ \rule{0.5em}{0.5em}}
\begin{document}

\title{Finitely generated groups with polynomial index growth}
\author{L\'{a}szl\'{o} Pyber\thanks{Supported in part by OTKA Grant no T 049841 } and
Dan Segal}
\maketitle

\begin{abstract}
We prove that a finitely generated soluble residually finite group has
polynomial index growth if and only if it is a minimax group. We also show
that if a finitely generated group with PIG is residually finite-soluble then
it is a linear group.\newline These results apply in particular to boundedly
generated groups; they imply that every infinite BG residually finite group
has an infinite linear quotient.

\end{abstract}

\section{Introduction}

A group $G$ is said to have \emph{polynomial index growth}, or \emph{PIG}, if
the orders of its finite quotients are polynomially bounded relative to their
exponents; that is, if there is a constant $\alpha$ such that%
\begin{equation}
\left\vert \overline{G}/\overline{G}^{n}\right\vert \leq n^{\alpha}%
\quad\forall n\in\mathbb{N} \label{PIGalpha}%
\end{equation}
for every finite image $\overline{G}$ of $G$ (and in this case we say that $G$
has PIG$(\alpha)$ ). The problem of characterizing finitely generated
residually finite (fg {\small R}$\mathfrak{F}$) groups with PIG was raised 20
years ago in [20], along with the analogous problem for polynomial subgroup
growth (PSG). While the latter was solved within eight years [11], the former
has proved resistant. For a detailed account of what is known about PIG and
its relation to other `upper finiteness conditions' see [12], Chapter 12.

The difficulty of the problem is explained by the far greater diversity of
groups with PIG. The fg {\small R}$\mathfrak{F}$ groups with PSG are precisely
those that are virtually soluble of finite rank. There are only countably many
of these, and they all have PIG; but many other groups do as well. Indeed, it
is shown in [2] that

\begin{itemize}
\item \emph{there exist uncountably many non-isomorphic fg }{\small R}%
$\mathfrak{F}$\emph{ groups with PIG.}
\end{itemize}

\noindent The construction in [2] gives groups that are neither virtually
soluble nor linear. Among the non-soluble linear groups there also exist many
groups with PIG; it was observed in [14] that arithmetic groups with the
congruence subgroup property (CSP) tend to have PIG. Subsequently, Lubotzky,
Platonov and Rapinchuk proved that PIG actually provides a group-theoretic
characterization of CSP:

\begin{itemize}
\item \emph{Let }$\Gamma$\emph{ be an }$S$\emph{-arithmetic group in a simply
connected, absolutely almost simple algebraic group over an algebraic number
field. Then }$\Gamma$\emph{ has PIG if and only if }$\Gamma$\emph{ has the
congruence subgroup property }([10], [17]).
\end{itemize}

Thus PIG may be seen as an `abstract' analogue to the congruence subgroup
property. The preceding result, however, suggests that a straightforward
characterization of \emph{all} fg {\small R}$\mathfrak{F}$ groups with PIG is
unlikely to be found. It is known [14] that every fg residually nilpotent
group with PIG is linear in characteristic zero, and it is not fanciful to
hope that the fg linear groups with PIG may be susceptible to some sort of
characterization. It is shown in [1] that those which are linear in
\emph{positive characteristic }are virtually abelian; but the general case
seems a distant prospect.

In this paper, we complete the work begun in [20] (and solve Problem 8e of
[12]) by proving

\begin{theorem}
\label{sol}Let $G$ be a fg {\small R}$\mathfrak{F}$ group that is virtually
soluble. Then $G$ has PIG if and only if $G$ has finite rank.
\end{theorem}

\noindent The fg {\small R}$\mathfrak{F}$ soluble groups of finite rank form a
well-understood class of groups, consisting of the fg soluble minimax groups
that are virtually torsion-free; some of their features are briefly recalled
in Section \ref{Hsection} below. (As usual, we say that a group is
\emph{virtually }$\mathcal{X}$ if it has a normal subgroup of finite index
with property $\mathcal{X}$.)

The theorem is actually a corollary of a more general result. For any group
$G$ we denote by $\mathrm{d}(G)$ the minimal cardinality of a generating set
for $G,$ and define%
\begin{equation}
\mathrm{rk}(G)=\sup\left\{  \mathrm{d}(H)\mid H\leq G\text{ with }%
\mathrm{d}(H)<\infty\right\}  ; \label{rank}%
\end{equation}
this is the \emph{rank} of $G$ (also called the Pr\"{u}fer rank). It is known
([14], Theorem 5.1) that every finite group of rank $r$ has PIG$(\alpha)$
where $\alpha\leq r(3+\log_{2}r)$; so a finite bound for the ranks of all
finite quotients of a group $G$ implies that $G$ has PIG, and the question is:
to what extent does the converse hold? The results stated above show that it
does not hold in general, even for finitely generated groups; our main result
answers the question as far as finite \emph{soluble} quotients are concerned.
Here $\exp(Q)$ denotes the exponent of a group $Q$.

\begin{theorem}
\label{thm1}Let $G$ be a finitely generated group, and let $\mathcal{S}$
denote the set of all finite soluble quotient groups of $G$. Suppose that
there exists $\alpha$ such that
\[
\left|  Q\right|  \leq\exp(Q)^{\alpha}
\]
for every $Q\in\mathcal{S}$. Then there is a finite upper bound for
$\mathrm{rk}(Q)$ as $Q$ ranges over $\mathcal{S}.$
\end{theorem}

\noindent The theorem can be formulated more picturesquely in the language of
profinite groups; this is discussed below. Combined with known results, it has
significant consequences for the global structure of fg {\small R}%
$\mathfrak{F}$ groups with PIG.

\begin{corollary}
\label{lin}Let $G$ be a fg group that is residually finite-soluble. If $G$ has
PIG then $G$ is a linear group over a field of characteristic $0.$
\end{corollary}

\noindent This follows from the theorem together with [14], Proposition 2.4.

The \emph{upper rank} of a group $G$ is the supremum of $\mathrm{rk}%
(\overline{G})$ over all finite quotients $\overline{G}$ of $G$. Theorem A of
[14] says that every fg {\small R}$\mathfrak{F}$ group of finite upper rank is
virtually soluble of finite rank (see [12] Chapter 5). With Theorem \ref{thm1}
this establishes

\begin{corollary}
\label{finrk}Let $G$ be a fg {\small R}$\mathfrak{F}$ group such that every
finite quotient of $G$ is soluble. Then\thinspace G has PIG if and only if $G$
is soluble of finite rank.
\end{corollary}

\noindent Since, as is easily seen, both finite rank and PIG are preserved on
passing to subgroups of finite index and to finite extensions, Theorem
\ref{sol} is essentially a special case of this corollary.

\bigskip

\textbf{Applications to bounded generation}

\medskip

A group is said to be \emph{boundedly generated} (BG) if it is equal to the
product of finitely many cyclic subgroups. This property is enjoyed by many
interesting groups: on the one hand, many higher-rank $S$-arithmetic groups
[27], on the other hand all fg soluble groups of finite rank [7].

It is easy to see that every BG group has PIG. So combining Kropholler's
theorem [7] with Theorem \ref{sol} we obtain

\begin{corollary}
Let $G$ be a fg {\small R}$\mathfrak{F}$ virtually soluble group. Then $G$ is
boundedly generated if and only if $G$ has finite rank.
\end{corollary}

\noindent Kropholler's theorem does not require residual finiteness; it would
be interesting to know if the converse also holds in this greater generality:

\begin{itemize}
\item \emph{Does every boundedly generated soluble group have finite rank?}
\end{itemize}

The main open problem regarding \emph{residually finite} BG groups is:
\emph{are they linear}? In the paper [26] this is attributed to Tavgen, in the
early '90s. Corollary \ref{lin} gives a positive answer in the special case of
groups that are residually finite-soluble; our next corollary is another
contribution to this question.

Before stating it, let us mention a useful condition that is in a sense
opposite to being virtually soluble: a group $G$ is said to have \emph{FAb} if
every virtually abelian quotient of $G$ is finite (this holds for example if
$G$ is an arithmetic subgroup of a simple algebraic group). A more
approachable version of Tavgen's question was posed by Bass:

\begin{itemize}
\item \emph{Is every} {\small R}$\mathfrak{F}$ BG \emph{group with} FAb
\emph{linear?}
\end{itemize}

The paper [1] establishes many properties of the BG groups with FAb. In
particular, the proof of Corollary 1.5 of [1] shows the following: if $G$ is
an infinite {\small R}$\mathfrak{F}$ BG group with FAb, then $G$ has an
infinite {\small R}$\mathfrak{F}$ quotient $Q$ such that \emph{either }(a) $Q$
is linear in characteristic zero \emph{or} (b) $Q$ has a normal subgroup
$Q_{0}$ of finite index such that every finite quotient of $Q_{0}$ is soluble.
Now Corollary \ref{lin} shows that case (b) is subsumed in case (a). Since
every fg virtually abelian group is also a linear group in characteristic
zero, we may infer

\begin{corollary}
\label{bglin}Every infinite {\small R}$\mathfrak{F}$ BG group has an infinite
linear image (in characteristic zero).
\end{corollary}

\noindent(It should be mentioned that this does not depend on the full force
of Theorem \ref{thm1}: Theorem \ref{T1}, stated below, suffices to show that
in case (b), $Q$ is virtually residually nilpotent, and hence linear by [14],
Theorem F.)

It is certainly necessary to assume residual finiteness in Corollary
\ref{bglin}: recently Muranov [13] has constructed an infinite simple group
which is equal to the product of 27 cyclic-subgroups. The following question
suggests itself:

\begin{itemize}
\item \emph{Does every }{\small R}$\mathfrak{F}$ BG \emph{group that is not
virtually soluble have a non virtually-soluble linear image?}
\end{itemize}

\noindent The answer is `no' if BG is replaced by PIG in this question: the
groups constructed in [2] have the property that every linear quotient is
virtually cyclic!\bigskip

\textbf{Further results}

\medskip

As usual in the investigation of `upper finiteness conditions', the proof of
the main theorem falls into two distinct parts. The first, `local', part
belongs to finite group theory and produces quantitative information. Here,
$f(\alpha)$ and $g(r,\alpha)$ denote certain functions of the exhibited arguments.

\begin{theorem}
\label{T1}Let $G$ be a finite soluble group with PIG$(\alpha)$. Then $G$ has a
nilpotent-by-metabelian normal subgroup of index at most $f(\alpha)$.
\end{theorem}

This will be used in conjunction with

\begin{proposition}
\label{rk(G/N')}Let $G$ be a finite soluble group and $N$ a nilpotent normal
subgroup of $G$. If $G$ has PIG$(\alpha)$ then
\[
\mathrm{rk}(G)\leq g(r,\alpha)
\]
where $r=\mathrm{rk}(G/N^{\prime})$.
\end{proposition}

\noindent($N^{\prime}$ denotes the derived group of $N$.)

These `finite' results can be formulated (in slightly weaker form) as
qualitative results about profinite groups. Routine arguments (left to the
reader) establish the following corollaries.

\begin{corollary}
\label{C1}Let $G$ be a prosoluble group with PIG. Then $G$ is virtually pronilpotent-by-metabelian.
\end{corollary}

\noindent(A \emph{prosoluble} (resp. \emph{pronilpotent}) group is an inverse
limit of finite soluble (resp. nilpotent) groups. A profinite group $G$ has
PIG if, for some $\alpha$, (\ref{PIGalpha}) holds for all \emph{continuous}
finite quotients $\overline{G}$ of $G$.)

\begin{corollary}
\label{C2}Let $G$ be a prosoluble group with PIG and $N$ a closed pronilpotent
normal subgroup of $G$. If $G/\overline{N^{\prime}}$ has finite rank then $G$
has finite rank.
\end{corollary}

\noindent Here, $\overline{N^{\prime}}$ denotes the closure of the derived
group $N^{\prime}$ of $N$. The rank of a profinite group $G$ is defined by
(\ref{rank}) where $H$ ranges over \emph{closed} -- or equivalently open --
subgroups and $\mathrm{d}(H)$ refers to \emph{topological} generators; cf. [3] \S 3.2.

Theorem \ref{thm1} can also be stated in `profinite language': it says that
\emph{the prosoluble completion of a fg group has PIG if and only if it has
finite rank}. But it cannot be established by `local' methods alone: one needs
the `global' hypothesis that the prosoluble group in question is the
completion of some finitely generated abstract group. Indeed \ Proposition 3.5
of [1] (see [12], Theorem 12.8.4) exhibits a metabelian finitely generated
profinite group $A$ with PIG that has infinite rank. This shows that Corollary
\ref{C1} is in a sense best possible.

The group $A$ is `boundedly generated' in the profinite sense (i.e. it is a
product of finitely many procyclic subgroups), and has relatively fast
subgroup growth. In Section \ref{exsec1} we show, by a similar construction,
that there exist $2$-generator metabelian profinite groups $B$ with
PIG$(3)$\ that have arbitrarily slow non-polynomial subgroup growth but
infinite rank (a prosoluble group of infinite rank cannot have PSG, by [12]
Theorem 10.2. It is also worth remarking that a fg \emph{abstract} metabelian
group, if its subgroup growth is not polynomial, must have subgroup growth of
type at least $2^{n^{1/d}}$ for some positive integer $d$ -- see [12],
\S \ 9.1). These groups $B$ have PIG also when viewed as (infinitely
generated) abstract groups, showing that the hypothesis of finite generation
is necessary in the main results stated above (it is easy to see, for example,
that $B$ is not linear, cf. [28], Theorem 2.2).

What the `finite' results do is to reduce the `global' problem to manageable
proportions, namely to establishing the following special case:

\begin{proposition}
\label{special}Let $E$ be a finitely generated soluble group of derived length
at most $3$. If $E$ has PIG then $E$ has finite upper rank.
\end{proposition}

\noindent The proof of this uses the theory of infinite soluble groups and
their group rings, initiated by P. Hall and J. E. Roseblade between 1950 and
1980, and further developed by Brookes and the second author. It has two
distinct components. The first is a general result of some independent
interest, about groups all of whose virtually residually nilpotent quotients
have finite rank (see Theorem \ref{H} below); this implies Proposition
\ref{special} in the `non-singular' case where certain primes do not appear in
the torsion of $E^{\prime\prime}$. The remaining `singular' case does not
succumb to such a crude attack, and has to be tackled directly; this is the
second component of the proof, and occupies Sections \ref{red1sec} and
\ref{singsec}.

Theorem \ref{H} also has applications to subgroup growth and other `upper
finiteness conditions', briefly discussed in Section \ref{Hsection}.

\bigskip

\textbf{Deduction of Theorem \ref{thm1}.}

\medskip

Let $G$ be a $d$-generator group satisfying the hypothesis of Theorem
\ref{thm1}. Put%
\begin{align*}
G_{0}  &  =\bigcap\left\{  N\vartriangleleft G\mid G/N\text{ is soluble and
}\left|  G:N\right|  \leq f(\alpha)\right\}  ,\\
E  &  =G_{0}/G_{0}^{\prime\prime\prime}.
\end{align*}
Then $\left|  G/G_{0}\right|  =m$, say, is finite, and it is easy to see that
every finite soluble quotient of $G_{0}$ has PIG$(\beta)$ where $\beta
=\alpha(1+\log_{2}m).$ Thus $E$ is finitely generated and has PIG$(\beta)$, so
Proposition \ref{special} shows that $E$ has finite upper rank, $r$ say.

Now let $Q=G/K$ be a finite soluble quotient of $G.$ Put $K_{0}=K\cap G_{0}$
and $N=G_{0}^{\prime\prime}K_{0}$. It follows from Theorem \ref{T1} that
$N/K_{0}$ is nilpotent. Now $G_{0}/N^{\prime}K_{0}=G_{0}/G_{0}^{\prime
\prime\prime}K_{0}$ is a finite quotient of $E$, hence has rank at most $r$.
Using Proposition \ref{rk(G/N')} we deduce that%
\begin{align*}
\mathrm{rk}(Q)=\mathrm{rk}(G/K)  &  \leq\mathrm{rk}(G/G_{0})+\mathrm{rk}%
(G_{0}/K_{0})\\
&  \leq\mathrm{rk}(G/G_{0})+g(r,\beta).
\end{align*}
As this holds for every $Q\in\mathcal{S}$ we see that Theorem \ref{thm1}
follows from Theorem \ref{T1}, Proposition \ref{rk(G/N')} and Proposition
\ref{special}. \medskip

We must conclude this introduction with an apology. The proof of Proposition
\ref{special} given in Sections \ref{Hsection} -- \ref{singsec} relies heavily
on the theory developed in [21], [22] and [24]. Rather than increase the bulk
of this paper by repeating the many definitions and results required, we will
simply quote as necessary; thus the reader who wants to work through the proof
in detail will need to have at least [21] and [24] to hand.

\section{Finite soluble groups\label{finitesec}}

In this section all groups are assumed to be \emph{finite} and \emph{soluble}.
We begin by recalling some well-known results about permutation groups and
linear groups.

\begin{proposition}
\label{PW}\emph{([16], [29]; see [15], \S 3) (i) }If $G$ is a primitive
permutation group of degree $n$ then%
\[
\left\vert G\right\vert <n^{13/4}.
\]
\emph{(ii) }If $G$ is a completely reducible subgroup of $\mathrm{GL}%
_{n}(\mathbb{F}_{p})$ then%
\[
\left\vert G\right\vert <p^{9n/4}.
\]

\end{proposition}

\begin{corollary}
\label{trans}If $T$ is a transitive permutation group of degree $t$ then%
\[
\exp(T)<t^{4}.
\]

\end{corollary}

\begin{proof}
$T$ is contained in a permutational wreath product $P\wr T_{2}$ where $P$ is
primitive of degree $t_{1}$, where $1<t_{1}\mid t$, and $T_{2}$ is transitive
of degree $t/t_{1}$. Inductively we may suppose that $\exp(T_{2}%
)<(t/t_{1})^{4}$, while Proposition \ref{PW}(i) shows that $\exp(P)<t_{1}^{4}%
$. The result follows.
\end{proof}

\bigskip

The next proposition, mainly due to Suprunenko, is taken from [25], p. 245
(here `primitive' is meant in the sense of linear groups):

\begin{proposition}
\label{sup}Let $S$ be maximal among the soluble primitive subgroups of
$\mathrm{GL}_{m}(\mathbb{F}_{p})$. Then $S$ has a unique maximal abelian
normal subgroup $A$. Let $C=\mathrm{C}_{S}(A)$ and $B=\mathrm{Fit}(C)$.
Then\newline\emph{(a) \ }$A$ is cyclic of order $p^{a}-1$, where $a\mid
m;$\newline\emph{(b) \ }$S/C$ is cyclic of order dividing $a$;\newline%
\emph{(c)} \ $\left|  B:A\right|  =b^{2}$ where $b=m/a$;\newline\emph{(d)}
\ $C/B$ is isomorphic to a completely reducible subgroup of the direct product
of symplectic groups%
\[
\prod_{i=1}^{k}\mathrm{Sp}_{2e_{i}}(\mathbb{F}_{q_{i}})
\]
where $b=\prod_{i=1}^{k}q_{i}^{e_{i}}$ and $q_{1},\ldots,q_{k}$ are the
distinct prime factors of $b$.
\end{proposition}

\begin{corollary}
\label{primlin}In the notation of the Proposition, we have%
\begin{align*}
\left|  C/A\right|  <b^{13/2},\\
\left|  S\right|  <p^{a}m^{7}.
\end{align*}

\end{corollary}

\begin{proof}
Applying Proposition \ref{PW}(ii) to the projections of $C/B$ in the product
(d) we see that%
\[
\left|  C/B\right|  <\prod_{i=1}^{k}q_{i}^{9e_{i}/2}=b^{9/2}.
\]
With (c) this gives the first claim, and the second claim follows since now%
\[
\left|  S\right|  <(p^{a}-1)b^{2}b^{9/2}a\leq p^{a}m^{7}.
\]

\end{proof}

\bigskip

Now let $G$ be a finite soluble primitive permutation group. Then $G=V\rtimes
H$ where $V$ is elementary abelian of order $p^{n}$, say, and the action of
$H$ on $V$ is faithful and represents $H$ as an irreducible subgroup of
$\mathrm{GL}_{n}(\mathbb{F}_{p})$. Decomposing $V$ into a maximal system of
imprimitivity for $H,$ we may embed $H$ in a permutational wreath product%
\begin{equation}
H\leq L\wr T \label{wp}%
\end{equation}
where $L$ is a primitive linear group of degree $m\mid n$ over $\mathbb{F}%
_{p}$ and $T$ is a transitive permutation group of degree $t=n/m$. Also $T$ is
an image of $H$. By Corollary \ref{trans} we have $\exp(T)<t^{4}$. Let $S$ be
a maximal primitive soluble subgroup of $\mathrm{GL}_{m}(\mathbb{F}_{p})$
containing $L$, and let $a$ be as in Proposition \ref{sup}. Then from
Corollary \ref{primlin} we have%
\[
\exp(L)\leq\left|  S\right|  <p^{a}m^{7}.
\]
It follows that%
\[
\exp(G)\leq p\exp(H)\leq p\exp(L)\exp(T)<p^{a+1}m^{7}t^{4}\leq p^{a+1}n^{7}.
\]

Assume now that $G$ has PIG($\alpha$). Then $\left|  T\right|  \leq
\exp(T)^{\alpha}<t^{4\alpha}$. Also%
\[
p^{n}\leq\left|  G\right|  \leq\exp(G)^{\alpha}<\left(  p^{a+1}n^{7}\right)
^{\alpha}.
\]
It follows that $n<(a+1)\alpha+7\alpha\log n/\log p,$ and hence that $n/a\leq
f_{1}=f_{1}(\alpha)$, a number depending only on $\alpha$.

In the notation of Proposition \ref{sup}, we now have $b=m/a\leq n/a\leq
f_{1}.$ As $S/C$ is cyclic, Corollary \ref{primlin} now implies that $S/A$ has
a cyclic subgroup of index at most $b^{13/2}<f_{1}^{7}$. Therefore $L$ has a
metacyclic subgroup $L_{1}$ of index $<f_{1}^{7}$. Now $H\leq L^{(t)}\cdot T$
where $L^{(t)}$ is the base group of the wreath product (\ref{wp}). Putting
$M=H\cap L_{1}^{(t)}$ we have $M^{\prime\prime}=1$ and
\[
\left|  H:M\right|  \leq\left|  L:L_{1}\right|  ^{t}\left|  T\right|
<f_{1}^{7t}t^{4\alpha}.
\]
Since $t=n/m\leq n/a\leq f_{1}$ it follows that $\left|  H:M\right|  \leq
f_{2}(\alpha)$, a number depending only on $\alpha$. Let $M_{0}$ be the
biggest normal subgroup of $H$ contained in $M$. Applying Corollary
\ref{trans} to the action of $H$ on the right cosets of $M$ we see that
$\exp(H/M_{0})<f_{2}(\alpha)^{4}$, and hence that $\left|  H/M_{0}\right|
<f_{2}(\alpha)^{4\alpha}=f_{3}(\alpha)$, say. Recalling that $H\cong G/V,$ we
may deduce

\begin{lemma}
\label{G0}Let $G$ be a finite soluble primitive permutation group with
PIG($\alpha$). Then $G$ has a normal subgroup $G_{1}$ such that $\left|
G/G_{1}\right|  <f_{3}(\alpha)$ and $G_{1}^{\prime\prime\prime}=1.$
\end{lemma}

It is now easy to complete the proof of\bigskip

\noindent\textbf{Theorem \ref{T1} \ }\emph{Let }$G$\emph{ be a finite soluble
group with PIG}$(\alpha)$\emph{. Then }$G$\emph{ has a nilpotent-by-metabelian
normal subgroup of index at most }$f(\alpha)$.

\bigskip

\begin{proof}
Put $h=f_{3}(\alpha)!$. Let $J$ be any maximal subgroup of $G$ and let $K$ be
the intersection of all $G$-conjugates of $J$. Then $G/K$ is a primitive
permutation group, so Lemma \ref{G0} shows that $(G^{h})^{\prime\prime\prime
}\leq K$. As the intersection of all such subgroups $K$ is the Frattini
subgroup $\Phi(G)$ of $G$, we have $(G^{h})^{\prime\prime\prime}\leq\Phi(G)$.
Put%
\[
G_{0}=G^{h}\Phi(G).
\]
Then $G_{0}^{\prime\prime}\Phi(G)/\Phi(G)$ is abelian, so $G_{0}^{\prime
\prime}$ is nilpotent by a well-known theorem of Gasch\"{u}tz. The result
follows on taking $f(\alpha)=h^{\alpha}$.
\end{proof}

\bigskip

Next we prove\bigskip

\noindent\textbf{Proposition \ref{rk(G/N')} \ }\emph{Let }$G$\emph{ be a
finite soluble group and }$N$\emph{ a nilpotent normal subgroup of }$G$\emph{.
If }$G$\emph{ has PIG}$(\alpha)$\emph{ then}
\[
\mathrm{rk}(G)\leq g(r,\alpha)
\]
\emph{where }$r=\mathrm{rk}(G/N^{\prime})$.

\bigskip

\begin{proof}
The proof is a slight variation on that of [12], Proposition 5.4.2 (an
argument originally due to A. Mann). Since $N$ is nilpotent we have
$\mathrm{rk}(N)=\max\,\mathrm{r}_{p}(N)$ where $\mathrm{r}_{p}(N)$ denotes the
rank of a Sylow $p$-subgroup of $N$, and $p$ ranges over primes dividing
$\left\vert N\right\vert $. Therefore $\mathrm{rk}(G)\leq r+\max
\,\mathrm{r}_{p}(N)$, and it will suffice to bound each $\mathrm{r}_{p}(N)$ in
terms of $r$ and $\alpha$. Fixing a prime $p$ we may now factor out
$O_{p^{\prime}}(G)$ and so assume that $O_{p^{\prime}}(G)=1$. Let $F$ denote
the Fitting subgroup of $G$. Then $F=O_{p}(G)$, and as $F\geq N$ we have
$\mathrm{rk}(F/F^{\prime})\leq r$. We will show that $\mathrm{rk}(F)$ is
bounded above by a function of $r$ and $\alpha$.

Write $F_{0}=F$ and for $i\geq0$ set $F_{i+1}=F_{i}^{\prime}F_{i}^{p}$. Then
$F_{0}/F_{1}\cong\mathbb{F}_{p}^{t}$ where $t\leq r$. It is well known and
easy to see that $G/F$ is then isomorphic to a completely reducible subgroup
of $\mathrm{GL}_{t}(\mathbb{F}_{p})$; this implies that $\left|  G/F\right|
\leq p^{3t}\leq p^{3r}$ by Proposition \ref{PW}(ii).

Let%
\[
s=\max\dim_{\mathbb{F}_{p}}(F_{i}/F_{i+1})
\]
and put $k=2+[\log_{2}s]$. Exercise 6 of [3], Chapter 2 shows that $F_{k}$ is
a powerful $p$-group, and then
\[
\mathrm{rk}(F_{k})=\dim_{\mathbb{F}_{p}}(F_{k}/F_{k+1})\leq s
\]
by [3] Theorem 2.9. It follows that $s=\dim_{\mathbb{F}_{p}}(F_{m}/F_{m+1})$
for some $m\leq k$.

Now the exponent of $G/F_{m+1}$ is at most%
\[
p^{m+1}\cdot\left|  G/F\right|  \leq p^{k+1+3r},
\]
which gives%
\[
p^{s}\leq\left|  G/F_{m+1}\right|  \leq p^{(k+1+3r)\alpha}.
\]
Hence $s\leq(\log_{2}s+3r+3)\alpha$, whence $s$ is bounded above in terms of
$r$ and $\alpha$. Since $\mathrm{rk}(F/F_{k})$ is at most $ks$ and
$\mathrm{rk}(F_{k})\leq s$ it follows that $\mathrm{rk}(F)$ is bounded above
as required.
\end{proof}

\section{Infinite soluble groups: the non-singular case\label{Hsection}}

Let us recall some basic facts about soluble groups of finite rank; the
results are principally due to Mal'cev, Robinson and Zaicev. A good reference
is Chapter 5 of [9]. A\ group $G$ has (Pr\"{u}fer) \emph{rank} $r$ if every
finitely generated subgroup of $G$ can be generated by $r$ elements, and $r$
is minimal with this property.

\begin{itemize}
\item \emph{A finitely generated soluble group} $G$ \emph{has finite rank if
and only if it is a} minimax \emph{group: that is, it possesses a finite
chain} $1=G_{0}\vartriangleleft G_{1}\vartriangleleft\ldots\vartriangleleft
G_{n}=G$ \emph{of subgroups such that each factor }$G_{i}/G_{i-1}$ \emph{is
either cyclic or quasi-cyclic} (of type $C_{p^{\infty}}$ for some prime $p$).
\end{itemize}

\noindent The finite set of primes $p$ for which such a $C_{p^{\infty}}$
factor occurs is the \emph{spectrum} $\mathrm{spec}(G)$.

For a minimax group $G$, the following conditions are equivalent:

\begin{itemize}
\item $G$\emph{ is virtually torsion-free,}

\item $\tau(G)$\emph{ is finite,}

\item $G$\emph{ is }{\small R}$\mathfrak{F}$ \ = \emph{residually finite,}

\item $G$\emph{ is reduced.}
\end{itemize}

\noindent Here $\tau(G)$ denotes the maximal periodic normal subgroup of $G$,
and $G$ is \emph{reduced} if $G$ contains no non-trivial radicable subgroups.

\begin{itemize}
\item \emph{The class of }{\small R}$\mathfrak{F}$ \emph{minimax groups is }extension-closed.

\item \emph{Every }{\small R}$\mathfrak{F}$ \emph{minimax group is
}({\small R}$\mathfrak{N}$)$\mathfrak{F}$ = \emph{virtually residually
nilpotent.}
\end{itemize}

Recall also that every fg ({\small R}$\mathfrak{N}$)$\mathfrak{F}$ group is
residually finite, since fg virtually nilpotent groups are residually finite
(cf. [9], \textbf{1.3.10}). This implies that for a \emph{finitely generated}
minimax group, \emph{the conditions} {\small R}$\mathfrak{F}$ \emph{and}
({\small R}$\mathfrak{N}$)$\mathfrak{F}$ \emph{are equivalent}.

We will also use without special mention the fact that

\begin{itemize}
\item \emph{a fg} {\small R}$\mathfrak{F}$ \emph{group has finite rank if and
only if it has finite upper rank};
\end{itemize}

\noindent this is [14] Theorem A (but we only apply it for the easier special
case of soluble groups).

\bigskip

A special case of our main theorem goes back to 1985 (see [12], Theorem 12.9):

\begin{theorem}
\emph{[20]} Let $G$ be a finitely generated group that is soluble and
({\small R}$\mathfrak{N}$)$\mathfrak{F}$. Then $G$ has PIG if and only if $G$
has finite rank.
\end{theorem}

This shows that fg soluble groups with PIG enjoy the following property:

\bigskip

\noindent\textbf{Definition} A group $G$ satisfies hypothesis $\mathcal{H}$ if
\emph{every }({\small R}$\mathfrak{N}$)$\mathfrak{F}$ \emph{quotient of} $G$
\emph{has finite upper rank}.

\bigskip

We shall prove

\begin{theorem}
\label{H}Let $E$ be a fg {\small R}$\mathfrak{F}$ group having a metabelian
normal subgroup $N$ such that $E/N$ is polycyclic. Suppose that $E$ satisfies
$\mathcal{H}$\textbf{.} Then $E/N^{\prime}$ is minimax, and if $N^{\prime}$
has no $\pi$-torsion where $\pi=\mathrm{spec}(E/N^{\prime})$ then $E$ is minimax.
\end{theorem}

\noindent This implies, for example, that every torsion-free fg {\small R}%
$\mathfrak{F}$ soluble group of derived length three that satisfies
$\mathcal{H}$ is a minimax group; in particular, it implies Proposition
\ref{special} for fg {\small R}$\mathfrak{F}$ groups that are torsion-free.

The condition on torsion is definitely necessary, however: in Section
\ref{exsec2} we exhibit a residually finite $3$-generator metabelian-by-cyclic
group that satisfies $\mathcal{H}$ but has infinite rank. This answers a
question posed in several places by the second author (see [23], Question 3;
[12] Introduction to Chapter 9).

Before proceeding, let us mention some other applications of Theorem \ref{H}.
For a prime $p$, the \emph{upper }$p$\emph{-rank} $\mathrm{ur}_{p}(G)$ of a
group $G$ is the supremum of $\mathrm{rk}(P)$ as $P$ ranges over all
$p$-subgroups of finite quotients of $G$. A theorem of Lucchini [8] and
Guralnick [4] (first proved for soluble groups by Kov\'{a}cs [6]) shows that
the upper rank $\mathrm{ur}(G)$ of $G$ satisfies%
\begin{equation}
\mathrm{ur}(G)\leq1+\sup_{p}\mathrm{ur}_{p}(G), \label{localrranks}%
\end{equation}
so $G$ has finite upper rank if and only if $\mathrm{ur}_{p}(G)$ is
\emph{bounded} as $p$ ranges over all primes. The following corollary of
Theorem \ref{H} establishes a special case of Conjecture A of [22]
(generalizing the main results of [22]):

\begin{corollary}
\label{urp}Let $G$ be a finitely generated group that is virtually
nilpotent-by-abelian-by-polycyclic. If $\mathrm{ur}_{p}(G)$ is finite for
every prime $p$, then $G$ has finite upper rank.
\end{corollary}

\begin{proof}
We may suppose wlog that $G$ has a nilpotent-by-abelian normal subgroup $N$
such that $G/N$ is polycyclic. Theorem 5 of [23] shows that a fg
({\small R}$\mathfrak{N}$)$\mathfrak{F}$ group having finite upper $p$-rank
for every prime $p$ has finite upper rank, so $G$ satisfies $\mathcal{H}.$ Now
$G/N^{\prime}$ is ({\small R}$\mathfrak{N}$)$\mathfrak{F}$ by [19], so it is
residually finite and therefore minimax$.$ Let $\pi=\mathrm{spec}(G/N^{\prime
})$. As $\pi$ is a finite set, it will suffice by (\ref{localrranks}) to show
that the numbers $\mathrm{ur}_{p}(G)$ are bounded as $p$ ranges over the set
$\pi^{\prime}$ of primes not in $\pi$.

Put%
\begin{align*}
\mathcal{X}  &  =\left\{  K\leq N^{\prime}\mid K\vartriangleleft
G\,\ \text{and}\,\ N^{\prime}/K\text{ is a finite }p\text{-group where }%
p\in\pi^{\prime}\right\}  ,\\
D  &  =\bigcap_{K\in\mathcal{X}}N^{\prime\prime}K.
\end{align*}
For each $K\in\mathcal{X}$ the group $G/N^{\prime\prime}K$ is again a
{\small R}$\mathfrak{F}$ minimax group, so $G/D$ is residually finite. It is
also metabelian-by-polycyclic and satisfies $\mathcal{H}$, and as $N^{\prime
}/D$ is residually a finite $\pi^{\prime}$-group, $N^{\prime}/D$ has no $\pi
$-torsion. It follows by Theorem \ref{H} that $G/D$ has finite rank, $r$ say.

Put $s=\sum_{i=1}^{c}r^{i}$ where $c$ is the nilpotency class of $N^{\prime}$.
Let $K\in\mathcal{X}$ and consider the finite $p$-group $P=N^{\prime}/K$. This
has nilpotency class at most $c$ and satisfies%
\[
\mathrm{rk}(P/P^{\prime})=\mathrm{rk}(N^{\prime}/N^{\prime\prime}K)\leq r;
\]
it follows that $\mathrm{rk}(P)\leq s$ (e.g. by [9], \textbf{1.2.11}). On the
other hand, as $N$ is nilpotent it is easy to see that if $p\in\pi^{\prime}$
then
\begin{align*}
\mathrm{ur}_{p}(G)  &  =\sup_{K\in\mathcal{X}}\mathrm{ur}_{p}(G/K)\\
&  \leq\mathrm{rk}(G/N^{\prime})+\sup_{K\in\mathcal{X}}\mathrm{rk}(N^{\prime
}/K)\leq\mathrm{rk}(G/N^{\prime})+s.
\end{align*}
The result follows.
\end{proof}

\bigskip

Together with Theorem 3 of [23] this implies that there is a `gap' in the
possible types of \emph{subgroup growth} for the groups in question:

\begin{corollary}
Let $G$ be a finitely generated group that is virtually
nilpotent-by-abelian-by-polycyclic. If $G$ has subgroup growth of type
strictly less than%
\[
n^{\log n/(\log\log n)^{2}}%
\]
then $G$ has finite upper rank, and hence has polynomial subgroup growth.
\end{corollary}

\noindent For the definition of `subgroup growth type' and further discussion,
see [23] and [12]. A direct application of Theorem \ref{H} together with [12],
Theorem 9.1 shows that there is a much larger gap in certain cases:

\begin{corollary}
Let $G$ be a fg {\small R}$\mathfrak{F}$ metabelian-by-polycyclic group. If
$G$ is torsion-free and has subgroup growth of type less than
$2^{n^{\varepsilon}}$ for every $\varepsilon>0$ then $G$ has finite rank,
hence polynomial subgroup growth.
\end{corollary}

The proof of Theorem \ref{H} will be completed in Sections \ref{redsec1} and
\ref{secP2}.

\section{Reductions\label{redsec1}}

Let us fix some notation, in force for the rest of the paper. We write%
\[
A\leq_{f}B,\,\,A\vartriangleleft_{f}B
\]
to mean `$A$ is a subgroup (resp. normal subgroup) of finite index in the
group $B$'. For a ring $R$,%
\[
J\vartriangleleft_{f}R,\,\,J\underset{\max}{\vartriangleleft}_{f}R
\]
means `$J$ is an ideal (resp. maximal ideal) of finite index in $R$'. For any
ideal $J$ of $R$ and any $R$-module $M,$ we write%
\[
MJ^{\infty}=\bigcap_{n=1}^{\infty}MJ^{n}.
\]

Let $G$ be a group and $k$ a commutative ring. The following special notation
is used for the group ring $kG$. If $H$ is any subgroup of $G$, the
corresponding l.c. Gothic letter $\mathfrak{h}$ will be used to denote the
right ideal%
\[
\mathfrak{h}=(H-1)kG\vartriangleleft kG.
\]
For an ideal $I$ of $kG$ we write%
\[
I^{\dagger}=(1+I)\cap G.
\]

All rings are assumed to have an identity $1\neq0$.

\bigskip

Throughout this section, $E$ will denote a finitely generated group that
satisfies $\mathcal{H}.$

\begin{lemma}
\label{L1}Let $D\vartriangleleft Q\vartriangleleft E$ where $E/Q$ is
polycyclic and $Q/D$ is nilpotent. Then\newline\emph{(i)} $Q/Q^{\prime}$ is
virtually torsion-free;\newline\emph{(ii)} $Q/D$ is a minimax group;\newline%
\emph{(iii)} $\mathrm{spec}(Q/D)\subseteq\mathrm{spec}(Q/Q^{\prime})=\pi$,
say;\newline\emph{(iv)} if $Q/D$ is torsion-free then $Q/D$ is residually a
finite $p$-group for every prime $p\notin\pi$.
\end{lemma}

\begin{proof}
Since $E/Q^{\prime}$ is abelian-by-polycyclic it is ({\small R}$\mathfrak{N}%
$)$\mathfrak{F}$ [19], hence minimax and residually finite. In particular,
$Q/Q^{\prime}$ is a virtually torsion-free minimax group. Claims (ii) and
(iii) follow by [9], \textbf{1.2.12}, and (iv) follows from [9],
\textbf{5.3.11}.
\end{proof}

\begin{proposition}
\label{P1}Suppose that $E$ has a normal subgroup $Q$ such that $E/Q$ is
polycyclic and $Q$ is residually (finite nilpotent). Then $E$ is minimax.
\end{proposition}

\begin{proof}
Lemma \ref{L1} shows that $Q/\gamma_{n}(Q)$ is minimax for every $n\geq2,$ and
that $\mathrm{spec}(Q/\gamma_{n}(Q))=\mathrm{spec}(Q/Q^{\prime})=\pi$, say, a
finite set of primes.\medskip

\textbf{Step 1.} For a prime $p$, let $Q(p)$ denote the finite-$p$ residual of
$Q$. Let $H/Q$ be a torsion-free nilpotent normal subgroup of $E/Q$, and put
$H_{0}=\mathrm{C}_{H}(Q/Q^{\prime}Q^{p})$. Then $H/H_{0}$ is finite and
$H_{0}(p)=Q(p)$.

$H/H_{0}$ is finite because $Q/Q^{\prime}Q^{p}$ is finite. Put $Q_{n}%
=\gamma_{n+1}(Q)Q^{p^{n}}$ for each $n$. Then $Q/Q_{n}$ is a finite $p$-group
and $H_{0}$ acts nilpotently on $Q/Q_{n}$, so $H_{0}/Q_{n}$ is a finitely
generated nilpotent group with torsion subgroup $Q/Q_{n}$. It follows that
$H_{0}/Q_{n}$ is residually a finite $p$-group ([9], \textbf{1.3.17}). Hence%
\[
H_{0}(p)\leq\bigcap_{n=1}^{\infty}Q_{n}=Q(p)
\]
and the claim follows.\medskip

\textbf{Step 2.} For each prime $p$ the group $E/Q(p)$ is minimax.

Since $E/Q$ is polycyclic, we may choose $H$ in Step 1 so that $E/H$ is
virtually abelian. Then there exists $K\vartriangleleft_{f}E$ with $K\geq
H_{0}$ such that $K/H_{0}$ is free abelian. Now apply Step 1 with $H_{0}$ in
place of $Q$ and $K$ in place of $H$: this shows that $K_{0}=\mathrm{C}%
_{K}(H_{0}/H_{0}{}^{\prime}H_{0}^{p})$ has finite index in $K$, hence also in
$E$, and that $K_{0}(p)=H_{0}(p)=Q(p)$. Thus $E/Q(p)$ is ({\small R}%
$\mathfrak{N}$)$\mathfrak{F}$ and the claim follows by hypothesis
$\mathcal{H}$.\medskip

\textbf{Step 3. }There exist a finite set of primes $\sigma$ and a
characteristic subgroup $T$ of $Q$ such that $Q/T$ is a torsion-free nilpotent
minimax group and $T/\gamma_{n}(Q)$ is a $\sigma$-group for all $n\geq n_{0}$
(where $n_{0}$ is the nilpotency class of $Q/T$).

Choose a prime $p\notin\pi=\mathrm{spec}(Q/Q^{\prime})$. For each $n$ let
$T_{n}/\gamma_{n}(Q)$ be the torsion subgroup of $Q/\gamma_{n}(Q).$ Lemma
\ref{L1}(iv) shows that $Q/T_{n}$ is residually a finite $p$-group, so we have%
\[
Q(p)\leq\bigcap_{n=1}^{\infty}T_{n}.
\]
In a minimax group, every descending chain of normal subgroup with
torsion-free factors is finite, so in view of Step 2 there exists $c$ such
that $T_{n}=T_{c}$ for all $n\geq c$. \ We put $T=T_{c}$. Then $T/[T,Q]$ is a
quotient of the periodic minimax group $T/\gamma_{c+1}(Q)$, hence $T/[T,Q]$ is
a $\sigma$-group where $\sigma$ is a finite set of primes. This implies that
$T/[T,\,_{k}Q]$ is a $\sigma$-group for every $k\geq1$ and hence that
$T/\gamma_{n}(Q)$ is a $\sigma$-group for every $n>c$.\medskip

\textbf{Conclusion.} Let $\mathcal{K}$ denote the set of all normal subgroups
$K$ of finite index in $Q$ such that $Q/K$ is nilpotent. For $K\in\mathcal{K}$
let $K^{\ast}/K$ be the $\sigma^{\prime}$-part of $Q/K$. Then $Q/K^{\ast}$ is
a finite nilpotent $\sigma$-group, and $T\cap K^{\ast}\leq K$. It follows that%
\[
T\cap\bigcap_{q\in\sigma}Q(q)\leq T\cap\bigcap_{K\in\mathcal{K}}K^{\ast}%
\leq\bigcap_{K\in\mathcal{K}}K=1
\]
since $Q$ is residually finite-nilpotent. Since $\sigma$ is finite and each of
the quotients $E/T,$ $E/Q(q)$ is minimax it follows that $E$ is minimax.
\end{proof}

\bigskip

Now suppose that $E$ is residually finite, and that $E$ has a metabelian
normal subgroup $N$ such that $E/N$ is polycyclic. Lemma \ref{L1} shows that
$N/N^{\prime}$ is virtually torsion-free and that $E/N^{\prime}$ is minimax;
we put $\pi=\mathrm{spec}(N/N^{\prime})$ and assume that $A=N^{\prime}$ has no
$\pi$-torsion. We write%
\[
G=N/N^{\prime},\,\,\Gamma=E/N^{\prime}%
\]
and consider $A$ as an (additively written) module for $\mathbb{Z}\Gamma$ and
its subring $R=\mathbb{Z}G$. To complete the proof of Theorem \ref{H} we have
to show that under these conditions, $E$ (or equivalently $A$) has finite rank.

\begin{lemma}
\label{L2}Suppose that $J$ is an ideal of $R$ containing $mR+\mathfrak{K}$
where $m\in\mathbb{N}$, $K=K^{\Gamma}\leq G$ and $\Gamma/K$ is polycyclic.
Then $A/AJ^{\infty}$ has finite rank as a $\mathbb{Z}$-module.
\end{lemma}

\begin{proof}
We may as well assume that $J=mR+\mathfrak{k}$, in which case $J=J^{\Gamma}$.
Say $K=Q/N^{\prime}$. Let $n\in\mathbb{N}$. Then $Q/AJ^{n}$ is nilpotent, so
by Lemma \ref{L1} it is minimax. The torsion subgroup of $Q/AJ^{n}$ has
exponent dividing $m^{n}t$, where $t$ is the exponent of the (finite) torsion
subgroup of $N/N^{\prime}$; hence $Q/AJ^{n}$ is finite-by-(torsion free), so
residually finite. It follows that $Q/AJ^{\infty}$ is residually
finite-nilpotent. Also $E/Q\cong\Gamma/K$ is polycyclic and $AJ^{\infty
}\vartriangleleft E$. Now Proposition \ref{P1} shows that $E/AJ^{\infty}$ is
minimax, and the lemma follows.
\end{proof}

\bigskip

Using Lemma \ref{L2}, in Section \ref{secP2} we will establish

\begin{proposition}
\label{P2}There exists a subgroup $H$ of $G$ such that $H\vartriangleleft
\Gamma$, $\Gamma/H$ is polycyclic, and for each finite $H$-module image
$A^{\ast}$ of $A$ there exists $n\in\mathbb{N}$ such that $A^{\ast
}\mathfrak{h}^{n}=0$.
\end{proposition}

Accepting this for now, we can complete the proof of Theorem \ref{H}. Say
$H=Q/N^{\prime}$. For $K\vartriangleleft_{f}E$ put $K_{1}=Q\cap K$. Then
$AK_{1}/K_{1}$ is a finite image of the $H$-module $A$, so Proposition
\ref{P2} implies that $Q$ acts nilpotently on $AK_{1}/K_{1}$. As $Q/A$ is
abelian it follows that $Q/K_{1}$ is nilpotent. Since $E$ is residually
finite, this shows that $Q$ is residually (finite nilpotent). As
$E/Q\cong\Gamma/H$ is polycyclic, Proposition \ref{P1} now shows that $E$ is a
minimax group.

\section{Quasi-fg modules\label{secqfg}}

Both in proving Proposition \ref{P2} and in later sections we make heavy use
of the machinery introduced in [22]. In that paper we considered an
abelian-by-polycyclic minimax group $\Gamma$ and \emph{finitely generated}
$\Gamma$-modules$.$ Here we have to deal with a module $A$ that may not be
finitely generated; instead, we are given that $A$ occurs in an exact sequence%
\[
1\rightarrow A\rightarrow E\rightarrow\Gamma\rightarrow1
\]
where the group $E$ is finitely generated. In this case, $A$ is said to be
\emph{quasi-fg}. In this section we show that such a module contains
nevertheless a finitely generated submodule $B$ that is relatively large, in
the sense that `most' finite quotients of $A$ are also quotients of $B$, and conversely.

\bigskip

Let $\Lambda$ be a non-empty subset of a ring $R$. The multiplicative
semigroup generated by $\Lambda$ is denoted $\left\langle \Lambda\right\rangle
$. An $R$-module $M$ is $\Lambda$\emph{-torsion} if $\mathrm{ann}_{R}(a)$
meets $\left\langle \Lambda\right\rangle $ for each $a\in M.$

\begin{lemma}
\label{jim}Let $R$ be a commutative ring and let $V\leq U$ be $R$-modules such
that $U/V$ is $\Lambda$-torsion, where $\varnothing\neq\Lambda\subseteq R$.
Let $J$ be an ideal of $R$ such that $J+\lambda R=R$ for all $\lambda
\in\Lambda$. Then for each $n\geq1$ we have%
\begin{align*}
UJ^{n}+V  &  =U\\
UJ^{n}\cap V  &  =VJ^{n}.
\end{align*}

\end{lemma}

\begin{proof}
Note that $J^{n}+\mu R=R$ for all $\mu\in\left\langle \Lambda\right\rangle $,
since any maximal ideal containing $\mu$ must meet $\Lambda$, and to simplify
notation we may as well take $n=1$. For $u\in U$ let $u^{\ast}$ denote the
annihilator in $R$ of $u+V\in U/V$. Then $u^{\ast}+J=R$, so%
\[
u\in u(u^{\ast}+J)\subseteq V+UJ.
\]

Now suppose $v\in UJ\cap V$. Then%
\[
v=u_{1}x_{1}+\cdots+u_{m}x_{m}%
\]
with $u_{i}\in U$ and $x_{i}\in J.$ There exists $\mu\in\left\langle
\Lambda\right\rangle \cap\prod u_{i}^{\ast}$, and then $\mu R+J=R.$ Hence%
\[
v\in\sum u_{i}\cdot\mu x_{i}+vJ\subseteq VJ.
\]
The result follows.
\end{proof}

\begin{proposition}
\label{quasifg}Let $E$ be a finitely generated group and $A\leq N$ normal
subgroups of $E$ such that $A$ is abelian, $N/A=\overline{N}$ is abelian of
finite rank and $E/N$ is polycyclic. Consider $A$ as an (additively written)
module for $E/A=\overline{E}$. Let $R$ denote the group ring $\mathbb{Z}%
\overline{N}$. Then $A$ contains a finitely generated $\overline{E}$-submodule
$B$ such that\newline\emph{(i)} $A/B$ is $(R\smallsetminus L)$-torsion
whenever $L$ is a maximal ideal of finite index in $R$ not containing
$\overline{\mathfrak{n}}=(\overline{N}-1)R$;\newline$\emph{(ii)}$ If
$\widetilde{A}$ is an $R$-quotient module of $A$ and $\widetilde{B}$ denotes
the image of $B$ in $\widetilde{A}$ then
\begin{align*}
\widetilde{A}J^{n}+\widetilde{B}  &  =\widetilde{A}\\
\widetilde{A}J^{n}\cap\widetilde{B}  &  =\widetilde{B}J^{n}%
\end{align*}
for every ideal $J$ of finite index in $R$ with $J+\overline{\mathfrak{n}}=R$
and every $n\in\mathbb{N}$.
\end{proposition}

\begin{proof}
Write $\overline{\phantom{x}}:E\rightarrow E/A$ for the quotient map. Since
$E$ is finitely generated and $E/N$ is a finitely presentable group, $N$ is
finitely generated as a normal subgroup of $E$. Also $\overline{N}$ is a
minimax group; we may therefore find a finitely generated subgroup $X$ of $N$
such that%
\begin{align*}
&  N=\left\langle X^{E}\right\rangle \\
&  \overline{N}/\overline{X}\text{ is a divisible torsion group.}%
\end{align*}
Say $X=\left\langle x_{1},\ldots,x_{s}\right\rangle $ and let $\{g_{1}%
=1,\,g_{2},\ldots,g_{d}\}$ be a finite subset of $E$ that generates $E$ as a
semigroup. We now specify a finitely generated $\mathbb{Z}\overline{E}%
$-submodule of $N^{\prime}$ by setting%
\[
B=\sum_{i,j,k}[x_{i},x_{j}^{g_{k}}]\cdot\mathbb{Z}\overline{E}%
\]
(we write the group operation in $A$ and $N^{\prime}\leq A$ additively).

Let $L$ be a maximal ideal of finite index in $R=\mathbb{Z}\overline{N}$ such
that $\overline{\mathfrak{n}}\nsubseteq L$. Put $\Lambda=R\smallsetminus L$.
We begin by showing that $N^{\prime}/B$ is $\Lambda$-torsion; this amounts to
establishing the\medskip

\noindent\textbf{Claim} For each $a\in N^{\prime}$ there exists $\lambda
\in\Lambda$ such that $a\lambda\in B$.\medskip

\noindent Now $N^{\prime}$ is generated as a normal subgroup of $E$ by
elements of the form $[x_{i},x_{j}^{w}]$ with $w\in E$. Thus $N^{\prime}$ is
additively generated by elements%
\[
a=[x_{i},x_{j}^{w}]\cdot\gamma,\,\,\,\gamma\in\overline{E}.
\]
As $\Lambda$ is multiplicatively closed, it will suffice to establish the
claim when $a$ is one of these.

This is done by induction on the length of $w$ as a (positive) word in
$g_{1},\ldots,g_{d}$. If this length is $1$ then $a\in B$ and we take
$\lambda=1$. Otherwise, $w=g_{k}v$ for some $k$ and $v$ an element of smaller
length. Using the Hall-Witt identity and the fact that $N$ is metabelian we
see that for any $z\in N$,%
\[
\lbrack\lbrack x_{i},x_{j}^{w}],z^{v}]=[[z,x_{j}^{g_{k}}]^{v},x_{i}%
]\cdot\lbrack\lbrack x_{i},z^{v}],x_{j}^{w}]
\]
which in additive notation becomes%
\begin{equation}
\lbrack x_{i},x_{j}^{w}]\cdot(\overline{z}^{\overline{v}}-1)=[z,x_{j}^{g_{k}%
}]\cdot\overline{v}(\overline{x_{i}}-1)+[x_{i},z^{v}]\cdot(\overline{x}%
_{j}^{\overline{w}}-1). \label{jacobi}%
\end{equation}

Now put $K=L^{\dagger(\overline{v}\gamma)^{-1}}$. Then $K$ is a proper
subgroup of finite index in $\overline{N},$ so cannot contain $\overline{X}$.
Thus for some $l\leq s$ we have $\overline{x_{l}}\notin K$, and the element%
\[
\mu=\overline{x_{l}}^{\overline{v}\gamma}-1
\]
lies in $R\smallsetminus L$. Taking $z=x_{l}$ in (\ref{jacobi}) we get%
\begin{align*}
a\mu &  =[x_{i},x_{j}^{w}]\cdot\gamma(\overline{x_{l}}^{\overline{v}\gamma
}-1)\\
&  =[x_{i},x_{j}^{w}]\cdot(\overline{x_{l}}^{\overline{v}}-1)\gamma\\
&  =[x_{l},x_{j}^{g_{k}}]\cdot\overline{v}(\overline{x_{i}}-1)\gamma
+[x_{i},x_{l}^{v}]\cdot(\overline{x}_{j}^{\overline{w}}-1)\gamma\\
&  =[x_{l},x_{j}^{g_{k}}]\cdot\overline{v}(\overline{x_{i}}-1)\gamma
+[x_{i},x_{l}^{v}]\cdot\gamma(\overline{x}_{j}^{\overline{w}\gamma}-1).
\end{align*}
The first term in the last line lies in $B$. Inductively, we may suppose that
there exists $\rho\in R\smallsetminus L$ such that $[x_{i},x_{l}^{v}%
]\cdot\gamma\rho\in B$. Then $\lambda=\mu\rho$ does the job, and the claim is established.

As $A\overline{\mathfrak{n}}\leq N^{\prime}$ and $\overline{\mathfrak{n}}$
meets $\Lambda$ we see that $A/N^{\prime}$ is also $\Lambda$-torsion. Hence
$A/B$ is $\Lambda$-torsion as required, and (i) follows.

Now let $J$ be an ideal of finite index in $R$ such that $\overline
{\mathfrak{n}}+J=R$. Then $J$ is contained in only finitely many maximal
ideals $L_{1},\ldots,L_{n}$, each of these has finite index and none of them
contain $\overline{\mathfrak{n}}$. Suppose $a\in A$ and write $a^{\ast
}=\mathrm{ann}_{R}(a+B)$. Then $a^{\ast}+L_{i}=R$ for each $i$, so $a^{\ast
}+J=R$. Thus $A/B$ and consequently also $\widetilde{A}/\widetilde{B}$ are
$\Lambda$-torsion where $\Lambda$ is the set $J+1$, and (ii) follows by Lemma
\ref{jim}.
\end{proof}

\section{Proof of Proposition \ref{P2}\label{secP2}}

Let us recall the setup from Section \ref{redsec1}. We have a residually
finite group $E$ and a metabelian normal subgroup $N$ of $E$ such that $E/N$
is polycyclic. Also $N/N^{\prime}$ is virtually torsion-free and $E/N^{\prime
}$ is minimax. We write%
\begin{align*}
G  &  =N/N^{\prime},\,\,\Gamma=E/N^{\prime}\\
\pi &  =\mathrm{spec}(G),
\end{align*}
and consider $A=N^{\prime}$ as an (additively written) module for
$\mathbb{Z}\Gamma$ and its subring $R=\mathbb{Z}G$.

We assume that the $R$-module $A$ has the following properties:

\begin{itemize}
\item $A$ has no $\pi$-torsion;

\item suppose that $J$ is an ideal of $R$ containing $mR+\mathfrak{K}$ where
$m\in\mathbb{N}$, $K=K^{\Gamma}\leq G$ and $\Gamma/K$ is polycyclic. Then
$A/AJ^{\infty}$ has finite rank as a $\mathbb{Z}$-module;
\end{itemize}

\noindent the second property was just the conclusion of Lemma \ref{L2}.

We have to prove that there exists a subgroup $H$ of $G$ such that
$H\vartriangleleft\Gamma$, $\Gamma/H$ is polycyclic, and for each finite
$H$-module image $A^{\ast}$ of $A$ there exists $n\in\mathbb{N}$ such that
$A^{\ast}\mathfrak{h}^{n}=0$.

\bigskip

We shall use without further ado the terminology and notation introduced in
[21]. Since $E$ is residually finite, $A$ is residually finite as a $\Gamma
$-module and \emph{a fortiori} as a $G$-module. Therefore $A$ is a \emph{qrf
}$G$-module\emph{\ }(`quasi-residually finite'). The hypothesis regarding
$\pi$-torsion says that $A$ is \emph{non-singular}. Let $B$ be the finitely
generated $\Gamma$-submodule of $A$ given in Proposition \ref{quasifg}, and
write%
\begin{align*}
A_{0}  &  =A(\mathfrak{g}):=\left\{  a\in A\mid a\mathfrak{g}^{n}=0\text{ for
some }n\geq1\right\}  ,\\
\widetilde{A}  &  =A/A_{0},\,\,\widetilde{B}=(B+A_{0})/A_{0}.
\end{align*}
We assume for now that $\widetilde{A}\neq0$. Recall that $\mathcal{P}(A)$
denotes the set of \emph{associated primes} of $A$: prime ideals of $R$ that
are annihilators of non-zero elements of $A$. Let%
\[
\mathcal{X}=\left\{  P\in\mathcal{P}(A)\mid P\supseteq\mathfrak{g}\right\}  .
\]
We claim that%
\[
A_{0}=A(\mathcal{X}):=\left\{  a\in A\mid aP_{1}\ldots P_{n}=0\text{ for some
}P_{1},\ldots,P_{n}\in\mathcal{X}\right\}  ;
\]
this is obvious if $\mathfrak{g}\in\mathcal{P}(A)$, and in general it follows
from [21], Cor. 6.6. Now [21], Lemma 6.5 shows that $\widetilde{A}$ is again
qrf and non-singular, and that $\mathcal{P}(\widetilde{A})\subseteq
\mathcal{P}(A)\smallsetminus\mathcal{X}$. Thus%
\[
P\in\mathcal{P}(\widetilde{A})\Longrightarrow\mathfrak{g}\nsubseteq P.
\]

Assume until further notice that $\widetilde{B}\neq0$. Let $\mathcal{Q}$
denote the set of minimal members of $\mathcal{P}(\widetilde{B})$, so
$\mathcal{Q}\subseteq\mathcal{P}(\widetilde{B})\subseteq\mathcal{P}%
(\widetilde{A})$ (and $\mathcal{Q}$ is non-empty because $\widetilde{B}%
\neq0).$ Now put%
\[
C=\widetilde{B}(\mathcal{P}(\widetilde{B})\smallsetminus\mathcal{Q}%
),\,\,M=\widetilde{B}/C.
\]
Then $\mathcal{P}(M)=\mathcal{Q}$ by [21], Lemma 6.5. Thus $M$ is an
\emph{unmixed }$G$-module, and (as a quotient of $B$) $M$ is finitely
generated as a module for $\mathbb{Z}\Gamma$.

Let $\left\{  P_{j}\mid j\in\mathcal{J}\right\}  $ be a set of representatives
for the orbits of $\Gamma$ in $\mathcal{Q}$. For each $j$ put $\Delta
_{j}=\mathrm{N}_{\Gamma}(P_{j})$ and $U_{j}=M/M(\mathcal{Q}\smallsetminus
\{P_{j}\})$. Now [21], Prop. 7.3 says the following: (i) the set $\mathcal{J}$
is finite, and for each $j\in\mathcal{J}$, (ii) $U_{j}$ is finitely generated
as a module for $\mathbb{Z}\Delta_{j}$ and (iii) there exists $e_{j}%
\in\mathbb{N}$ such that $U_{j}P_{j}^{e_{j}}=0$. Moreover, (iv) $M$ embeds
naturally into the direct sum of induced modules%
\[
\bigoplus_{j\in\mathcal{J}}U_{j}\uparrow_{\Delta_{j}}^{\Gamma},
\]
and projects onto each of the factors $U_{j}^{\gamma}$ ($\gamma\in\Gamma$).
Choosing each $e_{j}$ as small as possible, let $V_{j}=\mathrm{ann}_{U_{j}%
}(P_{j}^{e_{j}-1})$. Then%
\[
W_{j}=U_{j}/V_{j}%
\]
is a non-zero $P_{j}$-prime $R$-module and $W_{j}$ is finitely generated as
$\mathbb{Z}\Delta_{j}$-module. Since $\mathcal{J}$ is finite, each $W_{j}$ is
either $\mathbb{Z}$-torsion free or else has exponent $p$ for one of finitely
many primes $p$. The quotient maps $U_{j}\rightarrow W_{j}$ induce an exact
sequence of $\mathbb{Z}\Gamma$-modules%
\[
0\rightarrow D_{1}\rightarrow M\rightarrow\bigoplus_{j\in\mathcal{J}}%
W_{j}\uparrow_{\Delta_{j}}^{\Gamma}=W^{\ast},
\]
say, with $M$ projecting onto each of the factors $W_{j}^{\gamma}$ ($\gamma
\in\Gamma$).

Now the proof of [21], Theorem 8.5 shows that there exists a $\Gamma
$-invariant semimaximal ideal $J$ of finite index in $R$ such that
$(M/D_{1})J^{\infty}=0$. Moreover, since $\mathfrak{g}\nsubseteq P_{j}$ for
each $j$, we may choose$\ J$ to be an intersection of maximal ideals none of
which contains $\mathfrak{g}$; in this case we have $J+\mathfrak{g}=R$.

Recall that $M=\widetilde{B}/C$, so $D_{1}=D/C$ where $D$ is a $\mathbb{Z}%
\Gamma$-submodule of $\widetilde{B}$ and $\widetilde{B}J^{\infty}\leq D$. From
Proposition \ref{quasifg} we then have%
\[
\widetilde{A}J^{\infty}\cap\widetilde{B}=\widetilde{B}J^{\infty}\leq D.
\]
On the other hand, the initial hypothesis implies that $\widetilde
{A}/\widetilde{A}J^{\infty}$ has finite rank as a $\mathbb{Z}$-module. Hence
$\widetilde{B}/D$ has finite rank. But $\widetilde{B}/D\cong M/D_{1}$, and%
\[
M/D_{1}\geq\bigoplus_{P\in\mathcal{Q}}(M/D_{1})(\{P\})
\]
([21], Lemma 7.1). As at most finitely many primes occur in the torsion of
$M/D_{1}$, this implies that $\mathcal{Q}$ is finite, and hence that $W^{\ast
}$ has finite rank since each $W_{j}$ is an image of $M/D_{1}$. Therefore
$W^{\ast}$ decomposes as a $\Gamma$-module into a finite direct sum of finite
vector spaces over finite prime fields and a torsion-free $\mathbb{Z}$-module
of finite rank. By the Lie-Kolchin Theorem (cf. [28], Theorem 3.6), $\Gamma$
has a normal subgroup $\Gamma_{1}$ of finite index such that $\Gamma
_{1}^{\prime}$ acts unipotently on $W^{\ast}.$

Put $H=\Gamma_{1}^{\prime}\cap G$. Then $\Gamma/H$ is virtually polycyclic,
and there exists $n$ such that $W^{\ast}\mathfrak{h}^{n}=0$. Hence
$\mathfrak{h}$ annihilates each of the prime modules $W_{j}^{\gamma}$, and so
$\mathfrak{h}\subseteq P_{j}^{\gamma}$ for each $j\in\mathcal{J},\,\,\gamma
\in\Gamma$. Thus $\mathfrak{h}\subseteq P$ for every $P\in\mathcal{Q}$, and it
follows by [21], Lemma 6.5(iv) that $\widetilde{B}=\widetilde{B}%
(\mathfrak{h})$.

We claim now that%
\[
P\in\mathcal{P}(\widetilde{A})\Longrightarrow\mathfrak{h}\subseteq P.
\]
To see this, let $P\in\mathcal{P}(\widetilde{A})$. Then $P\nsupseteq
\mathfrak{g}$, so Theorem 4.2 of [21] shows that there exists a maximal ideal
$L$ of finite index in $R$ with $P\leq L$ and $\mathfrak{g}\nsubseteq L.$ Now
$P=\mathrm{ann}_{R}(a)$ for some $a\in\widetilde{A}$. According to Proposition
\ref{quasifg} there exists $\lambda\in R\smallsetminus L$ such that
$a\lambda\in\widetilde{B}$. Then $a\lambda\mathfrak{h}^{n}=0$ for some $n$; as
$P$ is a prime ideal and $\lambda\notin P$ it follows that $\mathfrak{h}%
\subseteq P$.

We have been assuming up to now that $\widetilde{B}\neq0$. The argument just
given shows that this is a \emph{consequenc}e of the assumption $\widetilde
{A}\neq0$: indeed, $\mathcal{P}(\widetilde{A})$ is non-empty, so choosing
$P\in\mathcal{P}(\widetilde{A})$ and $\lambda$ as above we see that $0\neq
a\lambda\in\widetilde{B}$.

Now $\widetilde{A}=\widetilde{A}(\mathcal{P}(\widetilde{A}))$, by [21], Lemma
6.5. Hence for each finite subset $S$ of $\widetilde{A}$ there exists
$n\in\mathbb{N}$ such that $S\mathfrak{h}^{n}=0$. This has been established
assuming that $\widetilde{A}\neq0$. It remains true (trivially) when
$\widetilde{A}=0$, in which case we take $H=G$.

\bigskip

\noindent\textbf{Conclusion} Let $A^{\ast}$ be a finite $H$-module image of
$A$, and write $A_{0}^{\ast}$ for the image of $A_{0}$ in $A^{\ast}.$ Then
$A^{\ast}/A_{0}^{\ast}$ is a finite image of $A/A_{0}=\widetilde{A}$, so there
exists $n_{1}$ such that $A^{\ast}\mathfrak{h}^{n_{1}}\leq A_{0}^{\ast}$. Also
$A_{0}^{\ast}$ is a finite image of $A_{0}=A(\mathfrak{g})$, so there exists
$n_{2}$ such that $A_{0}^{\ast}\mathfrak{g}^{n_{2}}=0$. Taking $n=n_{1}+n_{2}$
we have%
\[
A^{\ast}\mathfrak{h}^{n}=0.
\]
This completes the proof of Proposition \ref{P2}, and with it the proof of
Theorem \ref{H}.

\section{Proposition \ref{special}: a reduction\label{red1sec}}

We restate \bigskip

\noindent\textbf{Proposition \ref{special}} \emph{Let }$E$\emph{ be a finitely
generated soluble group of derived length at most }$3$\emph{. If }$E$\emph{
has PIG then }$E$\emph{ has finite upper rank.}

\bigskip

We observed above that this follows from Theorem \ref{H} in the case where $E$
is {\small R}$\mathfrak{F}$ and torsion-free. The possible presence of `bad
torsion' -- corresponding to primes in the spectrum of $E/E^{\prime\prime}$ --
means that we have to do a lot more work. We begin in this section by reducing
the problem to another special case.

\begin{proposition}
\label{p-singular}Let $E$ be a finitely generated group with normal subgroups
$N\geq A$ where $E/N$ and $N/A$ are abelian of finite rank, $A$ is an
elementary abelian $p$-group, and $A$ is residually finite-simple as a module
for $\mathbb{F}_{p}N$. If $E$ has PIG then $A$ is finite.
\end{proposition}

\noindent We make the convention that a \emph{simple} $\mathbb{F}_{p}N$-module
is assumed to be \emph{non-trivial for }$N$, so the hypothesis on $A$ amounts
to saying that%
\[
\bigcap_{L\in\mathcal{M}}AL=0,
\]
where $A$ is considered as an (additively-written) module for the abelian
group $N/A$ and where $\mathcal{M}$ is the set of all maximal ideals of finite
index in $\mathbb{F}_{p}(N/A)$ \emph{excluding} the augmentation ideal
$\mathbb{F}_{p}((N/A)-1)$.

The proof of this proposition occupies the next section. Assuming Proposition
\ref{p-singular} we now complete the

\bigskip

\noindent\textbf{Proof of Proposition \ref{special}.} We assume without loss
of generality that $E$ is residually finite. Write $N=E^{\prime}$ and
$A=N^{\prime}$, so $A$ is abelian. Since $E$ has PIG it satisfies hypothesis
$\mathcal{H}$; the first statement of Theorem \ref{H} shows that $E/A$ is a
minimax group, and we write $\pi=\mathrm{spec}(E/A)$.

For each prime $p$ let $r_{p}$ denote the upper $p$-rank of $A$ as $E$-module,
that is%
\begin{align*}
r_{p}  &  =\sup\left\{  \mathrm{rk}(A/C)\mid C\leq A,\,C\vartriangleleft
E\text{ and }A/C\text{ is a finite }p\text{-group}\right\} \\
&  =\sup\left\{  \mathrm{rk}(A/C)\mid A^{p}\leq C=C^{E}\leq_{f}A\right\}  .
\end{align*}
Then the upper rank of $E$ is at most $\mathrm{rk}(E/A)+\sup_{p}r_{p}$, and it
will suffice by (\ref{localrranks}) to show that the numbers $r_{p}$ are
boundedly finite as $p$ ranges over all primes.

Put%
\[
D_{0}=\bigcap\left\{  C\leq A\mid C\vartriangleleft E\text{ and }A/C\text{ is
a finite }\pi^{\prime}\text{-group}\right\}  .
\]
Then $E/D_{0}$ is residually finite and $A/D_{0}$ has no $\pi$-torsion.
Applying Theorem \ref{H} to the group $E/D_{0}$ we infer that $A/D_{0}$ has
finite rank $r$, say. It follows that $r_{p}\leq r$ for every prime
$p\notin\pi$.

Now $\pi$ is a finite set, so it remains to show that $r_{p}$ is finite for
each $p\in\pi$. Fix such a prime $p,$ write $k=\mathbb{F}_{p}$ and let%
\begin{align*}
D_{p}  &  =\bigcap\left\{  C\leq A\mid C\vartriangleleft E\text{ and
}A/C\text{ is a finite }p\text{-group}\right\}  ,\\
X_{p}  &  =\bigcap\left\{  T\leq A\mid A/T\text{ is a finite simple
}kN\text{-module}\right\}  .
\end{align*}
Proposition \ref{p-singular} applied to $E/X_{p}$ now shows that $A/X_{p}$ is
finite. This implies that $K:=\mathrm{C}_{N}(A/X_{p})$ has finite index in $N$.

We claim that $K/D_{p}$ is residually finite-nilpotent. To see this, consider
an arbitrary $E$-submodule $C$ of finite $p$-power index in $A$. Then $K$ acts
trivially on every simple quotient $kN$-module of $A/C$ \emph{including the
trivial ones}, hence $K-1\subseteq J$ where $J/\mathrm{ann}_{kN}(A/C)$ is the
Jacobson radical of the finite ring $kN/\mathrm{ann}_{kN}(A/C)$. It follows
that $(K-1)^{n}\subseteq\mathrm{ann}_{kN}(A/C)$ for some $n$, and hence that
$K/C$ is a nilpotent group. Also $E/C$ is residually finite, so $K/C$ is
residually finite-nilpotent. The claim follows since the submodules like $C$
intersect in $D_{p}$.

It now follows by Proposition \ref{P1} that $E/D_{p}$ has finite rank. Hence
$r_{p}$ is finite as required.

\section{The singular case\label{singsec}}

Now we begin the proof of Proposition \ref{p-singular}. Let $E$ be a finitely
generated group with PIG. We are given normal subgroups $N\geq A$ of $E$ such
that $E/N$ and $N/A$ are abelian of finite rank, $A$ is an elementary abelian
$p$-group, and $A$ is residually finite-simple as a module for $\mathbb{F}%
_{p}N$ (recall that `simple' means simple and non-trivial for $N$). The aim is
to prove that $A$ is finite; we assume that $A$ is infinite and propose to
derive a contradiction.

We will use without special mention the fact that every subgroup of finite
index in $E$ has PIG, and therefore satisfies $\mathcal{H}$, and we apply
results from Section \ref{redsec1}. Suppose that $N_{1}\leq_{f}N$ and
$N_{1}\vartriangleleft E.$ Then Lemma \ref{L1} shows that $N_{1}/[A,N_{1}]$ is
minimax, so $A/[A,N_{1}]$ is finite; if $D$ is the intersection of all
$kN_{1}$-submodules $C$ of $A$ such that $A/C$ is simple (non-trivial) for
$N_{1}$, then $D\cap\lbrack A,N_{1}]=0$, so $D$ is finite. Replacing $A$ by
$A/D$ we may therefore suppose that $A$ is residually finite-simple as a
$kN_{1}$-module.

We choose a specific $N_{1}$ as follows. Note that $E/A$ is virtually
torsion-free (e.g. by Lemma \ref{L1}). Let $E_{2}/A$ be a torsion-free normal
subgroup of finite index in $E/A$, put $N_{2}=N\cap E_{2}$ and $N_{0}%
=\mathrm{C}_{N_{2}}(A)$. Then $N_{2}/N_{0}$ is residually finite because $A$
is residually finite as $N$-module, and it follows that $E/N_{0}$ is virtually
torsion free. Let $E_{1}/N_{0}\vartriangleleft E/N_{0}$ be a torsion-free
subgroup of finite index in $E_{2}/N_{0}$, and put $N_{1}=N\cap E_{1}$.
Finally, let $E_{3}/N_{1}\vartriangleleft E/N_{1}$ be a torsion-free subgroup
of finite index in $E_{1}/N_{1}$. We may assume that $E_{2}$ and $E_{1}$ have
been chosen so as to minimize the rank of $N_{1}/N_{0}$; this rank is thus an
invariant of the triple $(E,N,A)$ and we denote it $s(E,N,A)$.

Replacing $E$ by $E_{3}$ and $N$ by $N_{1}$, we may thus assume that $E/N$,
$N/N_{0}$ and $N/A$ are all torsion free, where $N_{0}=\mathrm{C}_{N}(A)$. We
write%
\[
\Gamma=E/N_{0},\,\,G=N/N_{0}.
\]
Thus $G$ is a torsion-free abelian minimax group of rank $s(E,N,A)$ acting
faithfully on $A$; also $G$ is Noetherian as a module for $\Gamma$ (because
$\Gamma$ is finitely generated and $\Gamma/G$ is abelian, cf. proof of
Proposition \ref{quasifg}).

We write $k=\mathbb{F}_{p}$. As $A$ is infinite and residually simple for $kG$
we have $G\neq1$, so $s(E,N,A)\geq1$. We may assume that among all
counter-examples to Proposition \ref{p-singular}, the triple $(E,N,A)$ has
been chosen so that $s(E,N,A)$ is as small as possible.

\begin{lemma}
\label{K}If $1<K\leq G$ and $\left|  \Gamma:\mathrm{N}_{\Gamma}(K)\right|  $
is finite then $A/A\mathfrak{k}$ has finite upper rank as a module for
$\mathrm{N}_{\Gamma}(K)$.
\end{lemma}

\begin{proof}
Let $D$ denote the intersection of all maximal $kG$-submodules $C$ of $A$ such
that $A/C$ is non-trivial for $G$ but trivial for $K$. Say $\mathrm{N}%
_{\Gamma}(K)=E_{1}/N_{0}$. Then the triple $(E_{1}/D,N/D,A/D)$ satisfies the
hypotheses of Proposition \ref{p-singular}, and $s(E_{1}/D,N/D,A/D)\leq
\mathrm{rk}(G/K)<\mathrm{rk}(G)=s(E,N,A)$. Our inductive assumption thus
implies that $A/D$ is finite.

It follows that $Q:=\mathrm{C}_{N}(A/D)$ has finite index in $N$, so $E_{1}/Q$
is polycyclic. Let $P/A\mathfrak{k}$ be the finite residual of $E_{1}%
/A\mathfrak{k}$. The argument at the end of the preceding section shows that
$Q/P$ is residually finite-nilpotent, and Proposition \ref{P1} then shows that
$E_{1}/P$ has finite rank. The lemma follows.
\end{proof}

\begin{lemma}
\label{non-trivial}$\Gamma/\mathrm{C}_{\Gamma}(G)$ is infinite.
\end{lemma}

\begin{proof}
Suppose not. Then $E/N_{0}=\Gamma$ is virtually nilpotent, while $N_{0}$ is
nilpotent (of class at most 2). It follows by Lemma \ref{L1} that $N_{0}$ has
finite rank, which is impossible as $A\leq N_{0}$.
\end{proof}

\begin{lemma}
If $J\vartriangleleft_{f}kG$ then $A/AJ$ is finite.
\end{lemma}

\begin{proof}
There exists $K=K^{\Gamma}\leq_{f}G$ with $\mathfrak{k}\leq J$. Then
$K=Q/N_{0}$ where $E/Q$ is virtually nilpotent. As in the preceding lemma it
follows that $Q/[A,Q]$ has finite rank, and this implies the result since
$[A,Q]=A\mathfrak{k}\leq AJ$.
\end{proof}

\bigskip

Let us pause now to outline the strategy of the proof. The idea is to find a
`good' family of maximal ideals $L$ of finite index in $kG$, so that the
quotient modules $A/AL$ eventually get too big to fit inside the PIG group
$E$. In Subsection \ref{subs8.1} we establish arithmetical constraints that
must be satisfied by any `good' maximal ideal if $E$ has PIG($\alpha$).
Subsection \ref{subs8.2} considers a certain family $\mathcal{L}$ of maximal
ideals and shows that (almost always) if $L$ belongs to $\mathcal{L}$ then so
do all the `sisters' of $L$, namely the maximal ideals $M$ such that
$M^{\dagger}=L^{\dagger}$. The heart of the proof is in Subsection
\ref{subs8.3}: applying the structure theory from [21] and [24], we identify a
good family of maximal ideals, and pin down its properties; here Proposition
\ref{quasifg} is used. The results of all the preceding subsections are
combined to yield the final contradiction in Subsection \ref{subs8.4}.

To describe the arithmetic of maximal ideals we use the following notation:
let $L$ be a maximal ideal of finite index in $kX$ where $X$ is an abelian
group. Then $\left\vert kX/L\right\vert =p^{f}$ for some $f$, and $\left\vert
X/L^{\dagger}\right\vert =m$ where $p\nmid m$ and $f$ is equal to the order of
$p$ modulo $m$ (because $X/L^{\dagger}$ is a subgroup of $(kX/L)^{\ast}$ and
is contained in no proper subfield). We set $f(L)=f$ and $m(L)=m$.

Fix the parameters%
\[
d=\mathrm{d}(E)
\]
and $\alpha\in\mathbb{N}$ such that%
\[
\left\vert E/E^{n}\right\vert \leq n^{\alpha}\,\,\forall n\in\mathbb{N}%
\]
(note that $E/E^{n}$ is automatically finite since $E$ is finitely generated
and soluble).

\subsection{$f(L)$ versus $m(L)\label{subs8.1}$}

The next lemma will enable us to `lift' finite quotient modules of $A$ into
finite quotient groups of $E$.

\begin{lemma}
\label{lifting}Let $C<A$. Suppose that $A/C$ is a finite semisimple
$kG$-module and that $C\vartriangleleft E_{1}$ where $N<E_{1}\leq E$. Then
there exists $C^{\diamond}\vartriangleleft E_{1}$ such that $A\cap
C^{\diamond}=C$ and $AC^{\diamond}=\mathrm{C}_{E_{1}}(A/C)$.
\end{lemma}

\begin{proof}
Put $N_{1}=\mathrm{C}_{N}(A/C)$. Then%
\[
\lbrack N_{1}^{\prime},N]\leq\lbrack N,N_{1},N_{1}]\leq\lbrack A,N_{1}]\leq
C,
\]
so $N_{1}^{\prime}C/C$ is a trivial $G$-module. Therefore $N_{1}^{\prime}\leq
C$ since $A/C$ has no trivial $G$-submodules. Thus $N_{1}/C$ is abelian, and
putting $N_{2}=CN_{1}^{p}$ we have $N_{2}\cap A=C$ (recall that $N/A$ is torsion-free).

Now $N/N_{1}$ is a finite $p^{\prime}$-group since $A/C$ is semisimple for
$kG$, so $N_{1}/N_{2}$ is completely reducible as a module for $k(N/N_{1})$.
Therefore%
\[
\frac{N_{1}}{N_{2}}=\frac{AN_{2}}{N_{2}}\times\frac{N_{3}}{N_{2}}%
\]
where $N_{3}/N_{2}=\mathrm{C}_{N_{1}/N_{2}}(N)$. Note that $N_{1}/N_{3}\cong
AN_{2}/N_{2}\cong A/C$.

Put $D=\mathrm{C}_{E_{1}}(A/C)$. Then $D\cap N=N_{1}$ so
\[
\lbrack D^{\prime},N]\leq\lbrack N,D,D]\leq\lbrack N_{1},D]\leq N_{3},
\]
and hence $D^{\prime}\leq N_{3}$ since $D^{\prime}\leq N\cap D=N_{1}$ and
$N_{1}/N_{3}$ is semisimple for $G$. As above, it follows that $D^{p}N_{3}\cap
N_{1}=N_{3}$, and that%
\[
\frac{D}{D^{p}N_{3}}=\frac{D^{p}N_{1}}{D^{p}N_{3}}\times\frac{C^{\diamond}%
}{D^{p}N_{3}}%
\]
where%
\[
\frac{C^{\diamond}}{D^{p}N_{3}}=\mathrm{C}_{D/D^{p}N_{3}}(N).
\]

Looking back over the construction we find that%
\[
\frac{D}{C}=\frac{A}{C}\times\frac{C^{\diamond}}{C},
\]
and the result follows. (A lattice diagram makes the argument transparent.)
\end{proof}

\bigskip

The key to the inner structure of the relevant finite quotients is

\begin{lemma}
\label{lang}Let $1<V<P$ be normal subgroups of a finite group $H$, with
$V,\,P/V$ and $H/P$ abelian. Assume that $V=\mathrm{C}_{H}(V)$ and that $V$ is
semisimple and homogeneous as a $k(P/V)$-module. Let $C=\mathrm{C}_{H}(P/V)$
and put $r=\left|  H/C\right|  $. Then $\left|  P/V\right|  \cdot\left|
V\right|  ^{2/r}$ is an upper bound for the orders of all elements of $C/V$.
\end{lemma}

\begin{proof}
Write $\overline{\phantom{m}}:H\rightarrow \mathrm{Aut}(V)$ for the map
induced by conjugation in $H$ and write $F=k[\overline{P}]$ for the image of 
$k\overline{P}$ in $\mathrm{End}(V)$. Then $F\cong k\overline{P}/L$ where $L=%
\mathrm{ann}_{k\overline{P}}(V)$ is a maximal ideal of $k\overline{P}$, and $%
\overline{P}=\left\langle \zeta \right\rangle $ where $\zeta \in F^{\ast }$
has order $m=m(L)$. We have%
\begin{equation*}
(F:k)=f=f(L),\,\,V\cong F^{(n)} 
\end{equation*}%
for some $n\geq 1$; so $\left\vert V\right\vert =p^{nf}$, while $\left\vert
P/V\right\vert =\left\vert \overline{P}\right\vert =m$.

The action of $\overline{H}$ by conjugation gives a homomorphism $\theta :%
\overline{H}\rightarrow \mathrm{Gal}(F/k)$ with $\ker \theta =\overline{C}$,
so $\overline{H}=\overline{C}\left\langle x\right\rangle $ for some $x$;
note that $r=\left\vert \left\langle x\theta \right\rangle \right\vert $
divides $f=\left\vert \mathrm{Gal}(F/k)\right\vert $.

If $g\in \overline{C}$ then $[g,x]\in \overline{P}$ so%
\begin{equation*}
\lbrack g^{m},x]=[g,x]^{m}=1. 
\end{equation*}%
Writing $g^{m}$ as a product of a $p$-element and a $p^{\prime }$-element,
we see that it will suffice to prove\medskip

(a): if $y\in \overline{C}$ has $p$-power order then $y^{p^{n}}=1$;

\medskip

(b): if $y\in \mathrm{C}_{\overline{C}}(x)$ has order prime to $p$ then $%
y^{t}=1$ for some $t\leq p^{nf/r}$.

\medskip

\noindent \emph{Proof of} (a). This is clear since now $y$ is a unipotent $F$%
-automorphism of $V$.

\medskip

\noindent \emph{Proof of} (b). Let%
\begin{equation*}
R=k[\zeta ,y]\subseteq \mathrm{End}(V). 
\end{equation*}%
This is an image of the group algebra $F\left\langle y\right\rangle $ so it
is a semisimple $F$-algebra, on which $x$ induces via conjugation an
automorphism $\xi $ of order $r$. Let%
\begin{equation*}
J_{i}^{\xi ^{j}}\qquad (1\leq i\leq b,\,0\leq j<l_{i}) 
\end{equation*}%
be the distinct maximal ideals of $R$, grouped into $b$ $\left\langle \xi
\right\rangle $-orbits of lengths $l_{1},\ldots ,l_{b}$ respectively; put $%
e_{i}=(R/J_{i}:k)$. Then%
\begin{equation*}
\sum_{i=1}^{b}e_{i}l_{i}\leq \dim _{k}(V)=nf 
\end{equation*}%
since each $R/J_{i}^{\xi ^{j}}$ appears as a composition factor of the $R$%
-module $V$.

Now fix $i$ and put $E=R/J_{i}$. Then $\xi ^{l_{i}}$ acts on $E$, inducing
an automorphism of order exactly $r/l_{i}$ since $E$ contains (a copy of) $F$%
. If $\widetilde{y}=y+J_{i}$ then $\widetilde{y}$ lies in the fixed field of 
$\xi ^{l_{i}}$, so we have 
\begin{equation*}
(E:k(\widetilde{y}))\geq r/l_{i}. 
\end{equation*}%
Thus 
\begin{equation*}
(k(\widetilde{y}):k)\leq l_{i}(E:k)/r=e_{i}l_{i}/r. 
\end{equation*}%
Hence $y^{p^{e_{i}l_{i}/r}-1}-1\in J_{i}$.

Since $y=y^{\xi }$ and $\bigcap_{i,j}J_{i}^{\xi ^{j}}=0$ it follows that $%
y^{t}=1$ where%
\begin{equation*}
t=\prod_{i}\left( p^{e_{i}l_{i}/r}-1\right) <p^{\sum e_{i}l_{i}/r}\leq
p^{nf/r}. 
\end{equation*}
\end{proof}

\bigskip

Now we can establish the main result of this subsection:

\begin{proposition}
\label{m vs f}Suppose $\mathfrak{g}\neq L\underset{\max}{\vartriangleleft}%
_{f}kG$ satisfies $AL<A$, and that $L=L^{E_{1}}$ where $N<E_{1}\leq E$. Put
$s=\left|  E:E_{1}\right|  $, $K/N_{0}=L^{\dagger}$. Then at least one of the
following holds:\newline\emph{(a)}
\[
\left|  E_{1}:\mathrm{C}_{E_{1}}(N/K)\right|  \leq 4d\alpha\text{;}%
\]
\emph{(b)}
\[
\log_{p}m(L)>\frac{f(L)}{2(d+1)\alpha}-\frac{\log_{p}(psf(L))}{d+1}.
\]

\end{proposition}

\begin{proof}
Note that $\left\vert N/K\right\vert =m(L)$ and that $A/AL\cong (kG/L)^{(n)}$
for some $n\in \mathbb{N}$, so $\left\vert A/AL\right\vert =p^{nf(L)}$. Put $%
E_{2}=\mathrm{C}_{E_{1}}(N/K)$ and $r=\left\vert E_{1}:E_{2}\right\vert $.
Since $E_{1}/E_{2}$ acts faithfully by field automorphisms on $kG/L$, we
have $r\leq f(L)$.

According to Lemma \ref{lifting} there exists $C^{\diamond }\vartriangleleft E_{1}$
such that $A\cap C^{\diamond }=AL$ and $AC^{\diamond }=\mathrm{C}%
_{E_{1}}(A/AL)$. Note that $N\cap AC^{\diamond }=K$ and $NC^{\diamond }\leq
E_{2}$. Put%
\begin{align*}
H& =E_{1}/C^{\diamond } \\
C& =E_{2}/C^{\diamond } \\
P& =NC^{\diamond }/C^{\diamond } \\
V& =AC^{\diamond }/C^{\diamond }.
\end{align*}%
Then all conditions of Lemma \ref{lang} are fulfilled, and we infer that the
elements of $C/V$ have orders bounded by $\left\vert P/V\right\vert \cdot
\left\vert V\right\vert ^{2/r}=\left\vert N/K\right\vert \cdot \left\vert
A/AL\right\vert ^{2/r}=mp^{2nf/r}$ where $m=m(L)$ and $f=f(L)$. Since 
$H^{r}\leq C$, $C/P$ is abelian of rank at most $d$ and $P$ has exponent 
$mp$, it follows that the exponent of $H=E_{1}/C^{\diamond }$ is at most%
\begin{equation*}
r\cdot \left( mp^{2nf/r}\right) ^{d}\cdot mp=rm^{d+1}p^{2ndf/r+1}. 
\end{equation*}%
Hence $E^{sj}\leq E_{1}^{j}\leq C^{\diamond }$ for some $j\leq
rm^{d+1}p^{2ndf/r+1}$ and it follows that%
\begin{align*}
p^{nf}=\left\vert V\right\vert <\left\vert E:C^{\diamond }\right\vert & \leq
\left\vert E/E^{sj}\right\vert \leq (sj)^{\alpha } \\
& \leq (psr)^{\alpha }\cdot p^{2ndf\alpha /r}\cdot m^{(d+1)\alpha }.
\end{align*}

Suppose now that $r>4d\alpha $. Taking logarithms $\operatorname{mod}p$, noting that 
$r\leq f$ and rearranging we get%
\begin{equation*}
(d+1)\alpha \log _{p}m>nf/2-\alpha \log _{p}(psf). 
\end{equation*}%
This gives (b) since $n\geq 1$.
\end{proof}

\subsection{Good sisters\label{subs8.2}}

Here is a simple observation from [24]:

\begin{lemma}
\label{primes} \emph{([24], Lemma 12)} \ Let $\mathcal{X}$ be an infinite set
of maximal ideals of finite index in $kG$. Then $\mathcal{X}$ has an infinite
subset $\mathcal{Y}$ \ such that $\bigcap\mathcal{Y}$ is a prime ideal of $kG$.
\end{lemma}

This is used in the proof of

\begin{proposition}
\label{Land M}Let $\Gamma_{1}$ be a group acting on $G$ and let $0\neq\mu\in
kG$. Let%
\[
\mathcal{L}=\left\{  L\underset{\max}{\vartriangleleft}_{f}kG\mid\mu\notin
L=L^{\Gamma_{1}}\right\}  .
\]
Assume that $\bigcap\mathcal{X}=0$ for every infinite subset $\mathcal{X}$ of
$\mathcal{L}$. Then for almost all $L\in\mathcal{L}$ the following holds:%
\begin{equation}
M\underset{\max}{\vartriangleleft}_{f}kG,\,\,M^{\dagger}=L^{\dagger
}\Longrightarrow M\in\mathcal{L}. \label{sister}%
\end{equation}

\end{proposition}

\begin{proof}
For any $K<_{f}G$ let us write $\mathcal{M}_{K}=\left\{  M\underset{\max
}{\vartriangleleft}_{f}kG\mid\,M^{\dagger}=K\right\}  ,$ and put%
\[
\mathcal{K}=\left\{  L^{\dagger}\mid L\in\mathcal{L}\right\}  .
\]
For $s\in\mathbb{N}$ let $\psi_{s}$ denote the endomorphism of $G$ given by%
\[
\psi_{s}(g)=g^{s},
\]
and extend $\psi_{s}$ to a ring endomorphism of $kG$.

We begin with some observations.

\medskip

\textbf{1.} Let $K=L^{\dagger}\in\mathcal{K}$ where $L\in\mathcal{L}.$ Then
$\Gamma_{1}$ acts by field automorphisms on $kG/L$, so if $x\in\Gamma_{1}$ and
$g\in G$ then $g^{x}\equiv g^{p^{t}}\,(\operatorname{mod}K)$ for some $t$
(depending on $x$ but not on $g$). It follows that $\Gamma_{1}$ fixes every
ideal of $kG/\mathfrak{k}$, in particular every member of $\mathcal{M}_{K}$.

\medskip

\textbf{2.} Let $L,\,L^{\prime}\in\mathcal{M}_{K}$ where $K<_{f}G$. Then there
exists $s\in\mathbb{N}$ such that%
\[
L^{\prime}=\psi_{s}(L).
\]
To see this, note that $G/K$ is cyclic of order $m=m(L)=m(L^{\prime})$, and
that both $kG/L$ and $kG/L^{\prime}$ are isomorphic to $\mathbb{F}_{p^{f}}$
where $f$ is the order of $p$ modulo $m$. Choose an isomorphism $\theta
:kG/L\rightarrow kG/L^{\prime}$. As each of these fields contains a unique
multiplicative subgroup of order $m$, $\theta$ induces an automorphism
$\theta^{\ast}$ of $G/K:$%
\[%
\begin{array}
[c]{ccccc}%
G/K & \hookrightarrow & \left(  kG/L\right)  ^{\ast} & \subseteq & kG/L\\
\theta^{\ast}\downarrow &  &  &  & \downarrow\theta\\
G/K & \hookrightarrow & \left(  kG/L^{\prime}\right)  ^{\ast} & \subseteq &
kG/L^{\prime}%
\end{array}
\]
Then $\theta^{\ast}$ induces an automorphism $\widetilde{\theta}$ of
$kG/\mathfrak{k}\cong k(G/K)$ such that%
\[
\widetilde{\theta}(L/\mathfrak{k})=L^{\prime}/\mathfrak{k}.
\]
Now there exists $s$ coprime to $m$ such that $\theta^{\ast}(gK)=g^{s}K$ \ for
all $g\in G$. Then $\psi_{s}$ fixes $K$ and $\mathfrak{k}$ and induces
$\widetilde{\theta}$ on $kG/\mathfrak{k}$, so $\psi_{s}(L)=L^{\prime}$ as required.

\medskip

\textbf{3. }Let $M=M^{\Gamma_{1}}\underset{\max}{\vartriangleleft}_{f}kG$, let
$s\in\mathbb{N}$ and suppose that $\psi_{s}(\mu)\notin M$. Then $\widetilde
{M}:=\psi_{s}^{-1}(M)$ belongs to $\mathcal{L}$ and satisfies $\left\vert
kG/\widetilde{M}\right\vert ^{s^{r}}\geq\left\vert kG/M\right\vert $ where
$r=\mathrm{rk}(G)$.

Now $\psi_{s}$ maps $kG$ isomorphically onto $kG^{s}$ and $M\cap
kG^{s}\underset{\max}{\vartriangleleft}_{f}kG^{s}$, so%
\[
\widetilde{M}=\psi_{s}^{-1}(M\cap kG^{s})\underset{\max}{\vartriangleleft}%
_{f}kG
\]
and%
\[
\frac{kG}{\widetilde{M}}\cong\frac{kG^{s}+M}{M}\leq\frac{kG}{M}.
\]
As $\left\vert G/G^{s}\right\vert \leq s^{r}$ it follows that $\left\vert
kG/M\right\vert \leq\left\vert kG/\widetilde{M}\right\vert ^{s^{r}}$.
Obviously $\mu\notin\widetilde{M}$, and it is easy to see that $\widetilde
{M}=\widetilde{M}^{\Gamma_{1}}$. Thus $\widetilde{M}\in\mathcal{L}$.

\medskip

Suppose now that the proposition is false. Since each of the sets
$\mathcal{M}_{K}$ is finite, the set%
\[
\mathcal{K}_{0}=\left\{  K\in\mathcal{K}\mid\mathcal{M}_{K}\nsubseteq
\mathcal{L}\right\}
\]
must be infinite. Let $K\in\mathcal{K}$. By Observation 1, each member of
$\mathcal{M}_{K}$ is $\Gamma_{1}$-invariant, so if $K\in\mathcal{K}_{0}$ there
exists $M_{K}\in\mathcal{M}_{K}$ such that $\mu\in M_{K}$. By definition,
there exists $L_{K}\in\mathcal{M}_{K}\cap\mathcal{L}$, and Observation 2 shows
that $M_{K}=\psi_{s}(L_{K})$ for some $s\in\mathbb{N}$. Then $\psi_{s}%
(\mu)\notin M_{K}$. We fix such an $M_{K}$ for each $K\in\mathcal{K}_{0}$, and
write $s=s(K)$ for the appropriate $s$.

For each $s$ put%
\[
\mathcal{Y}(s)=\left\{  K\in\mathcal{K}_{0}\mid\psi_{s}(\mu)\in M_{K}\right\}
.
\]

\medskip

\textbf{Claim.} There is an infinite subset $\mathcal{X}$ of $\mathcal{K}_{0}$
such that%
\[
P:=\bigcap\mathcal{X}\text{ \ is prime}%
\]
and for $s\in\mathbb{N}$,%
\[
\psi_{s}(\mu)\notin P\Longrightarrow\left\vert \mathcal{X}\cap\mathcal{Y}%
(s)\right\vert <\infty.
\]

Suppose this is false. By Lemma \ref{primes} there is an infinite subset
$\mathcal{X}_{0}\subseteq\mathcal{K}_{0}$ such that%
\[
P_{0}:=\bigcap\mathcal{X}_{0}%
\]
is a prime ideal, and by hypothesis there exists $t\in\mathbb{N}$ such that
$\psi_{t}(\mu)\notin P_{0}$ and $\mathcal{K}_{1}:=\mathcal{X}_{0}%
\cap\mathcal{Y}(t)$ is infinite. Since $kG$ has finite Krull dimension, we may
choose $\mathcal{X}_{0}$ so as to maximize $P_{0}$. Repeating the preceding
step, we now find an infinite subset $\mathcal{X}\subseteq\mathcal{K}_{1}$
such that $P=\bigcap\mathcal{X}$ is prime. Then $P>P_{0}$ since $\psi_{t}%
(\mu)\in P\smallsetminus P_{0}$, so the maximal choice of $P_{0}$ implies that
$\mathcal{X}$ satisfies the \emph{claim}. Thus we have a contradiction.

\medskip

We can now conclude the proof. Choose $K_{1}\in\mathcal{X}$ and put
$s=s(K_{1})$. Then $\psi_{s}(\mu)\notin M_{K_{1}}$ so $\psi_{s}(\mu)\notin P$.
Therefore $\mathcal{X}\cap\mathcal{Y}(s)$ is finite, so $\mathcal{X}%
_{1}=\mathcal{X}\smallsetminus\mathcal{Y}(s)$ is infinite. If $K\in
\mathcal{X}_{1}$ then $\psi_{s}(\mu)\notin M_{K}$; let $\widetilde{L}%
_{K}=\widetilde{L}(M_{K},s)\in\mathcal{L}$ be the corresponding ideal provided
by Observation 3. As $K$ ranges over the infinite set $\mathcal{X}_{1}$ the
index $\left\vert kG:\widetilde{L}_{K}\right\vert $ is unbounded, so $\left\{
\widetilde{L}_{K}\mid K\in\mathcal{X}_{1}\right\}  $ is an infinite subset of
$\mathcal{L}$. It follows by hypothesis that%
\[
\bigcap_{K\in\mathcal{X}_{1}}\widetilde{L}_{K}=0.
\]
On the other hand, since $\mu\neq0$ and $kG$ is integral over $kG^{s}$ there
exists $\nu\neq0$ with $\nu\in kG^{s}\cap\mu kG^{s}$; thus $\nu=\psi_{s}%
(\tau)$ with $\tau\in kG$. For each $K\in\mathcal{X}_{1}$ we have $\psi
_{s}(\tau)\in\mu kG\leq M_{K}$, so $\tau\in\psi_{s}^{-1}(M_{K})=\widetilde
{L}_{K}$. Therefore $\tau=0$ and so $\nu=0$, a contradiction.
\end{proof}

\bigskip

Now, given $L\underset{\max}{\vartriangleleft}_{f}kG$, it is easy to estimate
how many `sisters' the ideal $L$ possesses. The following elementary
observation is proved in [24], Lemma 13:

\begin{lemma}
\label{cyclic}Let $C$ be a cyclic group of order $m$ where $p\nmid m$, and let
$f$ be the order of $p$ modulo $m$. Then the group algebra $kC$ has
$\varphi(m)/f$ faithful maximal ideals and each has index $p^{f}$.
\end{lemma}

\noindent Here $\varphi(m)$ denotes the Euler function. Since $G/L^{\dagger}$
is cyclic of order $m=m(L)$, it follows that $kG$ has $\varphi(m)/f$ maximal
ideals $M$ with $M^{\dagger}=L^{\dagger}$, where $f=f(L)$. Thus we have

\begin{corollary}
\label{manysisters}With $\mathcal{L}$ as above, let%
\[
\mathcal{L}(f)=\left\{  M\in\mathcal{L}\mid\left\vert kG/M\right\vert
=p^{f}\right\}  .
\]
If (\ref{sister}) holds for $L$ then%
\[
\left\vert \mathcal{L}(f)\right\vert \geq\varphi(m)/f
\]
where $m=m(L)$ and $f=f(L)$
\end{corollary}

\subsection{Finding good maximal ideals\label{subs8.3}}

For $\Delta\leq\Gamma$, $Q=Q^{\Delta}\vartriangleleft kG$ and $\lambda\in kG$,
we write%
\[
\mathcal{L}(\Delta,Q,\lambda)=\left\{  L\underset{\max}{\vartriangleleft}%
_{f}kG\mid Q\leq L,\,\lambda^{\delta}\notin L\,\,\forall\delta\in
\Delta\right\}  .
\]
We quote

\begin{proposition}
\label{ThmE} \emph{([24], Theorem 3)} Let $\Delta\leq\Gamma$, let
$Q=Q^{\Delta}$ be a prime ideal of $kG$ such that $G/Q^{\dagger}$ is reduced
and let $\lambda\in kG\smallsetminus Q$. Then%
\[
Q=\bigcap\mathcal{L}(\Delta,Q,\lambda).
\]

\end{proposition}

This implies that $\mathcal{L}(\Delta,Q,\lambda)$ is non-empty, and that if
$kG/Q$ is infinite then $\mathcal{L}(\Delta,Q,\lambda)$ is infinite.

\begin{proposition}
\label{prime-redn}There exist a faithful prime ideal $Q$ of $kG$, an element
$\lambda\in\mathfrak{g}\smallsetminus Q$ and a positive integer $\tau$ such
that\newline\emph{(i)} $\Delta:=\mathrm{N}_{\Gamma}(Q)$ has finite index in
$\Gamma$,\newline\emph{(ii)} $A$ has a $\Delta$-submodule $T$ such that%
\[
\frac{A}{T+AL}\cong\left(  \frac{kG}{L}\right)  ^{\tau}%
\]
as $kG$-module for every $L\in\mathcal{L}(\Delta,Q,\lambda)$.
\end{proposition}

\begin{proof}
By Proposition \ref{quasifg}, $A$ contains a finitely generated $k\Gamma
$-submodule $B$ such that $A/B$ is $(kG\smallsetminus L)$-torsion for every
$L\underset{\max}{\vartriangleleft}_{f}kG$ with $L\neq\mathfrak{g}$; and%
\[
B/BJ\cong A/AJ
\]
whenever $J$ is an ideal of finite index in $kG$ with $J+\mathfrak{g}=kG$
(this is an isomorphism of $k\Gamma_{1}$-modules if $\Gamma_{1}\leq\Gamma$
fixes $J$).

We use the terminology of [24]. Being residually finite-simple, the
$kG$-module $A$ is qrf and locally radical. Let $\mathcal{P}=\mathcal{P}((A))$
denote the set of associated primes of $A$, $\mathcal{Q}$ the set of minimal
members of $\mathcal{P}$ and $\mathcal{Y}=\mathcal{P}\smallsetminus
\mathcal{Q}$. Since $A$ is locally radical, each $P\in\mathcal{P}$ is the
annihilator of some non-zero element $a$ of $A$, so $kG/P\cong akG$ is again
qrf and $G/P^{\dagger}\leq\mathrm{Aut}_{kG}(akG)$ is reduced. Also
$\mathfrak{g}\notin\mathcal{P}$, since if $a\mathfrak{g}=0$ then%
\[
a\in a(\mathfrak{g}+L)\subseteq AL
\]
whenever $\mathfrak{g}\neq L\underset{\max}{\vartriangleleft}_{f}kG$; but the
intersection of all such submodules $AL$ is zero.

We claim that $\mathcal{P}=\mathcal{P}((B))$. To see this, let $P=\mathrm{ann}%
_{kG}(a)\in\mathcal{P}$. By [24], Theorem 1, $P$ is equal to an intersection
of maximal ideals of finite index; as $P\neq\mathfrak{g}$ we have $P\leq L$
for some such $L\neq\mathfrak{g}$. Then there exists $\lambda\in
kG\smallsetminus L$ such that $a\lambda=b\in B$, and then%
\[
bx=0\Longleftrightarrow\lambda x\in P\Longleftrightarrow x\in P
\]
so $P=\mathrm{ann}_{kG}(b)\in\mathcal{P}((B))$.

\medskip

Now set%
\[
B_{0}=B((\mathcal{Y})).
\]
Then $B/B_{0}$ is again qrf and locally radical, and $\mathcal{P}%
((B/B_{0}))=\mathcal{Q}$ (see [21], Lemma 6.5 and [24], Lemma 5). The results
of [24] combined with [21], Proposition 7.3 now show that the following
hold.\medskip

\textbf{(1) }The number of orbits of $\Gamma$ on $\mathcal{Q}$ is finite.

\medskip

(\textbf{2}) Let $P_{1},\ldots,P_{s}$ represent the distinct $\Gamma$-orbits
in $\mathcal{Q}$, put $\Delta_{i}=\mathrm{N}_{\Gamma}(P_{i})$ and let $Z_{i}$
be a transversal to $\Delta_{i}\backslash\Gamma$ for each $i$. Write%
\[
B_{i}/B_{0}=(B/B_{0})((\mathcal{Q}\smallsetminus\{P_{i}\})).
\]
Then $B/B_{i}$ is qrf and locally radical and $\mathcal{P}((B/B_{i}%
))=\{P_{i}\}$;

\medskip

\textbf{(3)} each $B/B_{i}$ is finitely generated as a module for $k\Delta
_{i}$.

\medskip

\textbf{(4)} $B/B_{0}$ embeds naturally as a subdirect sum in%
\[
\bigoplus_{i=1}^{s}\bigoplus_{z\in Z_{i}}\frac{B}{B_{i}z}.
\]

\medskip

We claim that $P_{i}$ has infinite index in $kG$ for at least one value of
$i$. Indeed, if not then $\mathcal{Q}$ consists of maximal ideals of $kG$, so
$\mathcal{Q}=\mathcal{P}$; and there exists $K=K^{\Gamma}\leq_{f}G$ with
$K\leq P^{\dagger}$ for every $P\in\mathcal{P}$. Results from [21] and [24]
show that if $a\in A$ then, writing $a^{\ast}=\mathrm{ann}_{kG}(a)$, we have%
\[
\sqrt{a^{\ast}}\supseteq Q_{1}\ldots Q_{n}%
\]
for some $Q_{1},\ldots,Q_{n}$ $\in\mathcal{P}$; this implies that%
\[
\mathfrak{k}\subseteq\sqrt{a^{\ast}}=a^{\ast}%
\]
since $A$ is locally radical. Thus $A\mathfrak{k}=0$ and $A=A/A\mathfrak{k}$
is finite, contradicting hypothesis.

We choose $i$ such that $kG/P_{i}$ is infinite, put $Q=P_{i}$, $\Delta
=\Delta_{i}$, $Z=Z_{i}$ and write $U=B/B_{i}$. According to [24], Proposition
4, $U$ is a prime $kG$-module with annihilator $Q$. By [22], Theorem 1.7,
there exist $\lambda\in kG\smallsetminus Q$ and a free $(kG/Q)$-submodule $F$
of $U$ such that every element of $U/F$ is annihilated by some product of
$\Delta$-conjugates of $\lambda$; without loss of generality we may take
$\lambda\in\mathfrak{g}$. Let $L\in\mathcal{L}(\Delta,Q,\lambda)$. Then%
\[
J=\bigcap_{\delta\in\Delta}L^{\delta}%
\]
is a $\Delta$-invariant ideal of finite index in $kG$ and $\lambda^{\delta
}kG+J=kG$ for each $\delta\in\Delta$. It follows that%
\[
\frac{U}{UJ}\cong\frac{F}{FJ}\cong\left(  \frac{kG}{J}\right)  ^{\tau}%
\]
where $\tau$ is the rank of the free $(kG/Q)$-module $U$.

Now put%
\[
J_{0}=\bigcap_{\gamma\in\Gamma}J^{\gamma}.
\]
Then $J_{0}$ is a $\Gamma$-invariant ideal of finite index in $kG$, so
$A/AJ_{0}$ is finite. Since $J_{0}+\mathfrak{g}=kG$ we have $B/BJ_{0}\cong
A/AJ_{0}$, so $B/BJ_{0}$ is finite. Let $S$ be a finite set of coset
representatives for $BJ_{0}$ in $B$. From (\textbf{4}) we see that $S$ is
contained in $B_{i}z$ for all but finitely many $z\in Z$, hence $B=BJ_{0}%
+B_{i}z$ for all but finitely many $z\in Z$. On the other hand, for each $z\in
Z$ we have%
\begin{align*}
\frac{B}{BJ_{0}+B_{i}z}  &  \cong\frac{B}{BJ_{0}+B_{i}}\text{ \ (}%
k\text{-module isomorphism)}\\
&  =\frac{U}{UJ_{0}}%
\end{align*}
and $U/UJ_{0}$ maps onto $U/UJ\cong(kG/J)^{\tau}$. As $\tau\geq1$ this implies
that $B>BJ_{0}+B_{i}z$; it follows that $Z$ is \emph{finite}. Thus $\Delta$
has finite index in $\Gamma$.

Since $B/BJ_{0}$ maps onto $(kG/J)^{\tau}$, it also follows that $\tau$ is finite.

Suppose that $Q^{\dagger}>1$. Then Lemma \ref{K} shows that $A/AQ$ has finite
upper rank $r$, say, as a $k\Delta$-module. Since $kG/Q$ is infinite, the set
$\mathcal{L}(\Delta,Q,\lambda)$ is infinite, and for each $L\in\mathcal{L}%
(\Delta,Q,\lambda)$ we have seen that $A/AL$ maps onto $B/BL\neq0$. Let
$\{L_{1},\ldots,L_{n}\}$ be a $\Delta$-invariant subset of $\mathcal{L}%
(\Delta,Q,\lambda)$, where $n>r$. Then by the Chinese Remainder Theorem%
\[
\frac{A}{AL_{1}\cap\ldots\cap AL_{n}}\cong\bigoplus_{j=1}^{n}\frac{A}{AL_{j}}%
\]
is a finite $k\Delta$-module quotient of $A/AQ$ of rank at least $n>r$, a
contradiction. It follows that $Q^{\dagger}=1$, and we have shown that $Q$ is faithful.

Now let $L\in\mathcal{L}(\Delta,Q,\lambda)$. Then $AL+B=A$ and $AL\cap B=BL$,
so by the modular law%
\[
\frac{A}{B_{i}+AL}\cong\frac{B}{B_{i}+BL}=\frac{U}{UL}\cong\left(  \frac
{kG}{L}\right)  ^{\tau}.
\]
Set $T=B_{i}$ to conclude the proof.
\end{proof}

\subsection{The final contradiction\label{subs8.4}}

We fix $Q,\,\Delta$ and $\lambda$ as above. Let%
\begin{align*}
\mathcal{L}  &  =\mathcal{L}(\Delta,Q,\lambda),\\
\mathcal{L}(f)  &  =\left\{  L\in\mathcal{L}\mid\left\vert kG/L\right\vert
=p^{f}\right\}
\end{align*}
for each $f\in\mathbb{N}$. Say $\Delta=E_{0}/N_{0}$ and put%
\[
t=\left\vert \Gamma:\Delta\right\vert =\left\vert E:E_{0}\right\vert .
\]
Recall that $\bigcap\mathcal{L}=Q$ and that $\mathcal{L}$ is an infinite set.

\begin{proposition}
\label{L(f)}There exists $n_{0}\in\mathbb{N}$ such that $\left|
\mathcal{L}(f)\right|  \leq n_{0}$ for all $f\in\mathbb{N}$.
\end{proposition}

\begin{proof}
Write $\overline{A}=A/T$. Let $L_{1},\ldots,L_{n}\in\mathcal{L}(f)$ be
distinct and suppose that $\Delta^{s}$ fixes each $L_{i}$. Then for each $i$,
the group $\Delta^{sf}$ acts trivially on the field $F_{i}=kG/L_{i}$, and
hence acts as an $F_{i}$-linear group on the $F_{i}$-vector space
$\overline{A}/\overline{A}L_{i}\cong F_{i}^{(h)}$.

It follows by Mal'cev's theorem ([28], Theorem 3.6) that there exists $m=m(h)$
such that $\Delta^{sfm}$ acts as a triangularizable group on $\overline
{A}/\overline{A}L_{i}$. Writing%
\[
q=p^{h}\prod_{j=1}^{h}(p^{jf}-1)
\]
we infer that $\Delta^{sfmq}$ acts as the identity on each $\overline
{A}/\overline{A}L_{i}$.

Now put $J=L_{1}\cap\ldots\cap L_{n}$ and $C=T+AJ$. Then%
\[
A/C\cong\bigoplus_{i=1}^{n}\overline{A}/\overline{A}L_{i}%
\]
is a semisimple $kG$-module, on which $E_{0}^{sfmq}$ acts as the identity.
According to Lemma \ref{lifting} there exists $C^{\diamond}\vartriangleleft
E_{0}$ such that $A\cap C^{\diamond}=C$ and $AC^{\diamond}=\mathrm{C}_{E_{0}%
}(A/C)$. Then%
\[
p^{hnf}=\left|  A/C\right|  =\left|  AC^{\diamond}/C^{\diamond}\right|
\leq\left|  E_{0}/C^{\diamond}\right|  .
\]
On the other hand, $E_{0}/AC^{\diamond}$ has exponent dividing $psfmq$, so%
\[
E^{tpsmfq}\leq C^{\diamond}.
\]
The original PIG hypothesis now implies that%
\[
p^{hnf}\leq\left(  tpsmfq\right)  ^{\alpha}.
\]
Noting that $q<p^{h(h+3)/2}$ and $f\geq1$ we deduce that%
\[
n\leq c_{1}+c_{2}\log s
\]
where $c_{1}$ and $c_{2}$ are constants that depend on $\alpha$, $p,$ $t$ and
$h$ but not on $f$ or $n$.

Suppose to begin with that $\{L_{1},\ldots,L_{n}\}$ is a single orbit of
$\Delta$. Since $\Delta/G$ is abelian, $\Delta^{n}$ now fixes each $L_{i}$, so
we can take $s=n$ in the above, and deduce that $n$ is absolutely bounded; say
$n\leq s_{0}$ for every such orbit. It follows that $\Delta^{s_{0}!}$ fixes
every member of $\mathcal{L}$, so taking $s=s_{0}!$ in the above we see that
$\mathcal{L}(f)$ contains at most $c_{1}+c_{2}\log s_{0}!$ distinct ideals.
\end{proof}

\bigskip

Our aim henceforth is to prove that, contrary to Proposition \ref{L(f)}, the
numbers $\left\vert \mathcal{L}(f)\right\vert $ are \emph{unbounded} as $f$
tends to infinity. This contradiction will complete the proof of Proposition
\ref{p-singular}

Proposition \ref{L(f)} implies that each orbit of $\Delta$ in $\mathcal{L}$
has length at most $n_{0}$, hence there exists $\Gamma_{1}\leq_{f}\Delta$ such
that $\Gamma_{1}$ fixes each ideal in $\mathcal{L}$. Note that $\Gamma_{1}$
has finite index in $\Gamma$. Let $Z$ be a transversal to the cosets of
$\Gamma_{1}$ in $\Delta$ and put $\mu=\prod_{z\in Z}\lambda^{z}$. Then%
\begin{equation}
\mathcal{L}=\mathcal{L}(\Delta,Q,\lambda)=\left\{  L\underset{\max
}{\vartriangleleft}_{f}kG\mid L=L^{\Gamma_{1}}\geq Q,\,\,\mu\notin L\right\}
. \label{L=l(Gamma)}%
\end{equation}

\bigskip

\textbf{Definition} Let $H\neq1$ be a torsion-free abelian group of finite
rank and $\Gamma$ a group acting on $H$. Then $(H,\Gamma)$ is a \emph{strict
Brookes pair} (for the prime $p$) if $\mathrm{spec}(H)\subseteq\{p\}$ and
there exist characters%
\[
\chi_{1},\ldots,\chi_{r}:\Gamma\rightarrow\left\langle p\right\rangle
<\mathbb{Q}^{\ast}%
\]
such that the quotient $H/(H_{\chi_{1}}\times\cdots\times H_{\chi_{r}})$ is a
finite $p^{\prime}$-group, where for a character $\chi$ we write%
\[
H_{\chi}=\left\{  g\in H\mid g^{\gamma}=g^{\chi(\gamma)}\,\,\forall\gamma
\in\Gamma\right\}  .
\]
The following result, essentially due to C. J. B. Brookes, is Theorem 2 of [24]:

\begin{proposition}
\label{control}Let $G$ be a torsion-free abelian minimax group acted on by a
group $\Gamma$ and let $P\neq0$ be a faithful $\Gamma$-invariant prime ideal
of $kG$. Then $G$ has a $\Gamma$-invariant subgroup $H$ such that $P=(P\cap
kH)kG$ and $(H,\Gamma_{0})$ is a strict Brookes pair for some $\Gamma_{0}%
\leq_{f}\Gamma$.
\end{proposition}

We shall say that $(G,\Gamma)$ is a \emph{virtually strict Brookes pair }if
there exist $G_{0}\leq_{f}G$ and $\Gamma_{0}\leq_{f}\Gamma$ such that
$(G_{0},\Gamma_{0})$ is a strict Brookes pair.

\begin{proposition}
\label{(a)(b)}At least one of the following holds:\newline\textbf{(a)}
$\bigcap\mathcal{X}=0$ for every infinite subset $\mathcal{X}$ of
$\mathcal{L}$;\newline\textbf{(b)} $(G,\Delta)$ is a virtually strict Brookes pair.
\end{proposition}

\begin{proof}
Suppose that (a) and (b) are both false. Then by Lemma \ref{primes}
$\mathcal{L}$ contains an infinite subset $\mathcal{Y}$ such that
$P=\bigcap\mathcal{Y}$ is a non-zero prime ideal, and we choose $\mathcal{Y}$
so that $P$ is as large as possible (as we may because $kG$ has finite Krull
dimension). Note that $P$ is $\Gamma_{1}$-invariant.

\medskip

\emph{Claim 1. }$P$ is faithful.

The proof is the same as the proof that $Q$ is faithful in the penultimate
paragraph of the proof of Proposition \ref{prime-redn}.

\medskip

\emph{Claim 2.} $P=(P\cap kH)kG$ where $H=H^{\Gamma_{1}}<G$ and $G/H$ is torsion-free.

According to Proposition \ref{control}, there exist $H_{0}=H_{0}^{\Gamma_{1}%
}\leq G$ and $\Gamma_{0}\leq_{f}\Gamma_{1}$ such that $P=(P\cap kH_{0})kG$ and
$(H_{0},\Gamma_{0})$ is a strict Brookes pair. Let $H$ be the isolator of
$H_{0}$ in $G$. Since $G$ is Noetherian as a $\Gamma$-module, $H/H_{0}$ is a
Noetherian $\Gamma_{0}$-module; as it is also a periodic abelian minimax
group, $H/H_{0}$ is finite. Since we are supposing (b) false, it follows that
$H\neq G$.

\medskip

\emph{Claim 3.} Fix $M\in\mathcal{Y}$ and put $T=(M\cap kH)kG$. Then
$T=\bigcap\mathcal{Z}$ for some subset $\mathcal{Z}$ of $\mathcal{L}$.

To see this, note that $T$ is a prime ideal of $kG$, because $G/H$ is
torsion-free; also $G/T^{\dagger}=G/(H\cap M^{\dagger})$ is reduced and
$\mu\notin T=T^{\Gamma_{1}}$, so by Proposition \ref{ThmE} we have%
\[
T=\bigcap\mathcal{L}(\Gamma_{1},T,\mu).
\]
Since $M\cap kH\geq P\cap kH$ we also have $T\geq P\geq Q$, so $\mathcal{L}%
(\Gamma_{1},T,\mu)\subseteq\mathcal{L}(\Gamma_{1},Q,\mu)=\mathcal{L}$.

\medskip

\emph{Conclusion.} Since $P\neq0$ the group $H$ is infinite, so $T^{\dagger
}=H\cap M^{\dagger}$ is infinite. As $P\leq T$ and $P$ is faithful it follows
that $P<T$; by the maximal choice of $P$ this forces $\mathcal{Z}$ to be a
finite set. Hence $kG/T$ is finite, a contradiction since $G/H$ is infinite.
\end{proof}

\bigskip

We examine the two cases in turn.

\medskip

\textbf{Case (a) }Suppose that $\bigcap\mathcal{X}=0$ for every infinite
subset $\mathcal{X}$ of $\mathcal{L}$. Then $Q=0$ since $Q=\bigcap\mathcal{L}$
and $\mathcal{L}$ is infinite.

Let $\mathcal{L}_{1}$ denote the set of $L\in\mathcal{L}$ such that%
\[
M\underset{\max}{\vartriangleleft}_{f}kG,\,\,M^{\dagger}=L^{\dagger
}\Longrightarrow M\in\mathcal{L}.
\]
Proposition \ref{Land M}, with (\ref{L=l(Gamma)}), shows that $\mathcal{L}%
\smallsetminus\mathcal{L}_{1}$ is finite.

Suppose $L\in\mathcal{L}_{1}$ with $m(L)=\left\vert G/L^{\dagger}\right\vert
=m$ and $f(L)=\log_{p}\left\vert kG/L\right\vert =f$. Then according to
Corollary \ref{manysisters} we have%
\[
\left\vert \mathcal{L}(f)\right\vert \geq\varphi(m)/f;
\]
with Proposition \ref{L(f)} this gives
\begin{equation}
\varphi(m)\leq n_{0}f. \label{ineq}%
\end{equation}

Let $\mathcal{L}_{2}$ denote the set of $L\in\mathcal{L}_{1}$ such that%
\[
\left\vert \Gamma_{1}:\mathrm{C}_{\Gamma_{1}}(G/L^{\dagger})\right\vert
>4d\alpha.
\]
We claim that $\mathcal{L}_{1}\smallsetminus\mathcal{L}_{2}$ is finite. To see
this, let $\Gamma_{2}$ be the intersection of all subgroups of index at most
$4d\alpha$ in $\Gamma_{1}$. Then $\Gamma_{2}$ has finite index in $\Gamma$
while%
\[
\lbrack G,\Gamma_{2}]\leq\bigcap_{L\in\mathcal{L}_{1}\smallsetminus
\mathcal{L}_{2}}L^{\dagger}.
\]
Lemma \ref{non-trivial} shows that $[G,\Gamma_{2}]\neq1$; but if
$\mathcal{L}_{1}\smallsetminus\mathcal{L}_{2}$ is infinite then the
intersection on the right is equal to $1$. The claim follows.

Now let $L\in\mathcal{L}_{2}$. Then from Proposition \ref{m vs f} we have%
\[
\log_{p}m(L)>\frac{f(L)}{2(d+1)\alpha}-\frac{\log_{p}(psf(L))}{d+1}%
\]
where $s=\left|  \Gamma:\Gamma_{1}\right|  $. With (\ref{ineq}) this implies
that $m(L)$ is bounded above by a constant (depending on $p,\,n_{0},\,s,\,d$
and $\alpha$) -- note that $\varphi(m)>\sqrt{m}$ for large $m$. Since
$G/G^{m}$ is finite for each $m$ it follows that $\mathcal{L}_{2}$ is finite.

Hence $\mathcal{L}$ is finite, a contradiction.

\medskip

\textbf{Case (b) \ }We quote

\begin{proposition}
\emph{([24], Proposition 6) }Let $(G,\Delta)$ be a virtually strict Brookes
pair and let $Q=Q^{\Delta}$ be a faithful prime ideal of infinite index in
$kG$. Let $\lambda\in kG\smallsetminus Q$. Then the numbers $\left\vert
\mathcal{L}(\Delta,Q,\lambda)(f)\right\vert $ are unbounded as $f\rightarrow
\infty$.
\end{proposition}

Thus if $(G,\Delta)$ is a virtually strict Brookes pair we again have a
contradiction to Proposition \ref{L(f)}.

As each case of Proposition \ref{(a)(b)} leads to a contradiction, the proof
is complete.

\section{Profinite Examples\label{exsec1}}

Let $f:\mathbb{N}\rightarrow\mathbb{R}_{>0}$ be a non-decreasing function with
$f(n)\rightarrow\infty$ as $n\rightarrow\infty$. We construct sequences
$(p_{i})$, $(m_{i})$ and $(q_{i})$ recursively as follows.

Let $p_{1}\equiv1\,(\operatorname{mod}4)$ be a prime such that $f(p_{1})>4$,
put $m_{1}=1$ and set $q_{1}=(p_{1}+1)/2$.

Given $p_{j}$, $m_{j}$ and $q_{j}$ for $1\leq j<i$, set $P_{i}=\prod
_{j<i}p_{j},\,Q_{i}=\prod_{j<i}q_{j}$, and let $p_{i}$ be a prime such that%
\begin{align*}
f(p_{i})  &  \geq2+2m_{j}\,\,\forall j<i\\
p_{i}  &  \equiv1\,(\operatorname{mod}4P_{i}Q_{i}).
\end{align*}
The existence of such a prime is guaranteed by Dirichlet's theorem on
arithmetic progressions. Now set%
\begin{align*}
m_{i}  &  =[\frac{f(p_{i})-2}{2}]\\
q_{i}  &  =\frac{p_{i}^{m_{i}}+1}{2}.
\end{align*}
Let $F_{i}$ be the field of size $p_{i}^{2m_{i}}$; the multiplicative group of
$F_{i}$ contains an element $x_{i}$ of order $q_{i}$. One checks easily that
$p_{i}$ and $q_{i}$ are relatively prime to each other and to $P_{i}Q_{i}$,
and that $x_{i}$ is a primitive element for $F_{i}$. It follows that the
semi-direct product%
\[
B_{i}=F_{i}\rtimes\left\langle x_{i}\right\rangle
\]
(where $F_{i}$ is taken as an $\mathbb{F}_{p}\left\langle x_{i}\right\rangle
$-module) is generated by the two elements $\mathbf{e}_{i}=(1_{F_{i}},1)$ and
$\mathbf{x}_{i}=(0,x_{i})$.

Now consider the profinite group%
\[
B=\prod_{i=1}^{\infty}B_{i}.
\]
Noting that the $q_{i}$ and the $p_{i}$ all all relatively prime, it is easy
to see that $B$ is generated topologically by the two elements $\mathbf{e}%
=(\mathbf{e}_{i})$ and $\mathbf{x}=(\mathbf{x}_{i})$.

We claim that $B$ has PIG$(3)$. Since $F_{i}$ is the unique minimal normal
subgroup of $B_{i}$ and the groups $B_{i}$ have pairwise coprime orders, any
finite quotient of $B$ takes the form%
\begin{equation}
Q=\prod_{j\in X}B_{j}\times\prod_{j\in Y}C_{j} \label{prod}%
\end{equation}
where $X$ and $Y$ are disjoint finite sets of indices and $C_{j}$ is some
quotient of $\left\langle x_{j}\right\rangle $. As the factors in (\ref{prod})
have pairwise coprime orders, it will suffice to show that each factor has
PIG$(3)$. This is obvious for the $C_{j}$; as for $B_{j}$ we have (writing
$p=p_{j},\,q=q_{j},\,m=m_{j}$)%
\begin{align*}
\left|  B_{j}\right|   &  =\frac{1}{2}p^{2m}(p^{m}+1)\\
\exp(B_{j})  &  =\frac{1}{2}p(p^{m}+1)
\end{align*}
so $\left|  B_{j}\right|  <\exp(B_{j})^{3}$.

As $m_{i}\rightarrow\infty$ with $f(p_{i})$ the rank of $B$ is infinite. On
the other hand, the slow growth of the $m_{i}$ limits the subgroup growth of
$B$. For each prime $p$ let $r_{p}$ denote the upper $p$-rank of $B$, and
denote by $s_{n}(B)$ the number of open subgroups of idex at most $n$ in $B$.
The proof of Lemma 1.7.1 of [12] shows that for every $n,$%
\[
s_{n}(B)\leq n^{2+h(n)}%
\]
where%
\[
h(n)=\max\left\{  r_{p}\mid p\leq n\right\}  .
\]
Now if $p=p_{i}$ for some $i$ then $r_{p}=2m_{i}$; if $p$ is any other prime
then $r_{p}\leq1.$ Therefore $h(n)$ is bounded above by%
\[
\max\left\{  2m_{i}\mid p_{i}\leq n\right\}  \leq\max\left\{  f(p_{i})-2\mid
p_{i}\leq n\right\}  \leq f(n)-2.
\]
Hence $s_{n}(B)\leq n^{f(n)}$ for all $n$.

\section{ A finitely generated example\label{exsec2}}

There are a number of theorems that characterize the fg residually finite
groups that satisfy some upper finiteness condition as being those that are
virtually soluble minimax groups. Many of them were first proved, or have only
been proved, under the additional assumption that the groups are virtually
residually nilpotent. Some examples were mentioned in Section \ref{Hsection},
and several more appear in [12]. These motivated the question:

\begin{itemize}
\item \emph{If a fg} \emph{abelian-by-minimax group is residually finite, is
it necessarily} ({\small R}$\mathfrak{N}$)$\mathfrak{F}$?
\end{itemize}

Here we present an easy counterexample, which also shows that Theorem \ref{H}
fails if the hypothesis on torsion is removed.

Fix a prime $p$ and let $G$ be the additive group of $\mathbb{Z}[\frac{1}{p}%
]$, written multiplicatively:%
\[
G=\left\langle x_{n}\,(n\in\mathbb{Z})\,;\,x_{n}^{p}=x_{n-1}\,\forall
n\right\rangle .
\]
Let $\tau$ be the automorphism of $G$ sending $x_{n}$ to $x_{n+1}$ for each
$n$. Then $\tau$ extends to an automorphism of the group ring $A=\mathbb{F}%
_{p}G,$ and hence to an automorphism of the wreath product%
\[
W=C_{p}\wr G=A\rtimes G.
\]
Now let%
\[
E=W\rtimes\left\langle \tau\right\rangle .
\]

\begin{proposition}
\emph{(i) }$E$ is a $3$-generator residually finite abelian-by-minimax group
of infinite upper rank.\newline\emph{(ii)} Every ({\small R}$\mathfrak{N}%
$)$\mathfrak{F}$ quotient of $G$ is a minimax group.
\end{proposition}

\begin{proof}
The group $E$ is clearly generated by $e,\,x_{0}$ and $\tau$ where $e$
generates the `bottom group' $C_{p}$ of $W$. Also $E/A\cong G\rtimes
\left\langle \tau\right\rangle $ is minimax.

For $m\in\mathbb{N}$ put
\begin{align*}
A_{m}  &  =[A,G^{m}]\\
&  =(G^{m}-1)\mathbb{F}_{p}G.
\end{align*}
Then $A/A_{m}\cong\mathbb{F}_{p}(G/G^{m})$, and if $p\nmid m$ then $G/G^{m}$
is cyclic of order $m$, in which case $A/A_{m}$ has rank $m$.

Now $AG^{m}/A_{m}$ is locally cyclic modulo its centre, so it is abelian. As
$E/AG^{m}$ is virtually cyclic it follows that $E/A_{m}$ is virtually
metabelian, hence residually finite ([9], \S 4.3). In particular, $E$ has a
normal subgroup $N$ of finite index with $N\cap A=A_{m}$. Hence%
\[
\mathrm{ur}(E)\geq\mathrm{rk}(E/N)\geq\mathrm{rk}(A/A_{m}).
\]
Letting $m$ range over all numbers prime to $p$ we see that $E$ has infinite
upper rank.

Since the subgroups $G^{m}$ intersect in $\{1\}$, the ideals $(G^{m}%
-1)\mathbb{F}_{p}G$ intersect in zero in the ring $\mathbb{F}_{p}G$, whence
$\bigcap_{m}A_{m}=1$. As each of the quotients $E/A_{m}$ is residually finite
it follows that $E$ is residually finite.

(ii) Suppose that $N\leq K$ are normal subgroups of $E$ such that $E/K$ is
finite and $K/N$ is residually nilpotent. There exists $m$ such that
$G^{m}\leq K$, and then $A_{m}=[A,G^{m}]\leq K$. Now the ideal $(G^{m}%
-1)\mathbb{F}_{p}G$ is idempotent in the ring $\mathbb{F}_{p}G$ since $G^{m}$
is $p$-divisible; this means that $A_{m}=[A_{m},G^{m}]$. Hence $A_{m}$ is
contained in every term of the lower central series of $K$, and so $A_{m}\leq
N$. Thus $E/N$ is a quotient of the minimax group $E/A_{m}.$
\end{proof}

\bigskip

\begin{center}
{\Large References}
\end{center}

[1] M. Ab\'{e}rt, A. Lubotzky and L. Pyber, Bounded generation and linear
groups, \emph{Int. J. Algebra and Computation} \textbf{13} (2003), 401-413.

\bigskip

[2] A. Balog, A. Mann and L. Pyber, Polynomial index growth groups, \emph{Int.
J. Algebra and Computation} \textbf{10} (2000), 773-782.

\bigskip

[3] J. D. Dixon, M. P. F. du Sautoy, A. Mann and D. Segal, \emph{Analytic
pro-}$p$\emph{ groups, 2nd edn.}, Cambridge Studies Adv. Math. \textbf{61},
CUP, Cambridge, 1999.

\bigskip

[4] R. M. Guralnick, On the number of generators of a finite group,
\emph{Archiv der Math.} \textbf{53} (1989), 521-523.

\bigskip

[5] M. V. Horoshevskii, On automorphisms of finite groups (Russian),
\emph{Mat. Sbornik} \textbf{93} (1974), 576-587.

\bigskip

[6] L. Kov\'{a}cs, On finite soluble groups, \emph{Math. Zeit.} \textbf{103}
(1968), 37-39.

\bigskip

[7] P. H. Kropholler, On finitely generated soluble groups with no large
wreath product sections, \emph{Proc. London Math. Soc.}(3) \textbf{49} (1984), 155-169.

\bigskip

[8] A. Lucchini, A bound on the number of generators of a finite group,
\emph{Archiv der Math.} \textbf{53} (1989), 313-317.

\bigskip

[9] J. C. Lennox and D. J. S. Robinson, \emph{The theory of infinite soluble
groups}, Clarendon Press, Oxford, 2004.

\bigskip

[10] A. Lubotzky, Subgroup growth and congruence subgroups, \emph{Invent.
Math. }\textbf{119} (1995), 267-295.

\bigskip

[11] A Lubotzky, A. Mann and D. Segal, Finitely generated groups of polynomial
subgroup growth, \emph{Israel J. Math.} \textbf{82} (1993), 363-371.

\bigskip

[12] A Lubotzky and D. Segal, \emph{Subgroup Growth,} Progress in Math.
\textbf{212}, Birkh\"{a}user, Basel, 2003.

\bigskip

[13] A. Muranov, Diagrams with selection and method for constructing
bounded-generated and bounded-simple groups, \emph{preprint: }www.arxiv

\bigskip

[14] A. Mann and D. Segal, Uniform finiteness conditions in residually finite
groups, \emph{Proc. London Math. Soc. }(3) \textbf{61 }(1990), 529-545.

\bigskip

[15] O. Manz and T. R. Wolf, \emph{Representations of solvable groups,} LMS
Lect. Notes \textbf{185}, CUP, Cambridge, 1993.

\bigskip

[16] P. P. P\'{a}lfy, A polynomial bound for the orders of primitive solvable
groups, \emph{J. Algebra} \textbf{77} (1982), 127-137.

\bigskip

[17] V. P. Platonov and A. S. Rapinchuk, Abstract properties of $S$-arithmetic
groups and the congruence subgroup problem, \emph{Izv. Ross.} \emph{Akad. Nauk
Ser. Mat.}\textbf{56} (1992), 483-508 (Russian); \emph{Russian Acad. Sci. Izv.
Math.} \textbf{40} (1993), 455-476 (English).

\bigskip

[18] V. P. Platonov and A. S. Rapinchuk, \emph{Algebraic groups and number
theory}, Academic Press, San Diego, 1994.

\bigskip

[19] D. Segal, On abelian-by-polycyclic groups, \emph{J. London Math. Soc.}
(2)\textbf{ 11 }(1975), 445-452.

\bigskip

[20] D. Segal, Subgroups of finite index in soluble groups, I and II, pp.
307-319 in \emph{Groups St Andrews 1985}, LMS Lect. Notes \textbf{121}, CUP,
Cambridge, 1986.

\bigskip

[21] D. Segal, On the group rings of abelian minimax groups, \emph{J. Algebra}
\textbf{237} (2001), 64-94.

\bigskip

[22] D. Segal, On modules of finite upper rank, \emph{Trans. AMS} \textbf{353}
(2000), 391-410.

\bigskip

[23] D. Segal, On the finite images of infinite groups, pp. 542-563 in
\emph{Groups: topological, combinatorial and arithmetic aspects}, LMS Lect.
Notes \textbf{121}, CUP, Cambridge, 2004.

\bigskip

[24] D. Segal, On the group rings of abelian minimax groups, II: the singular
case,  \emph{J. Algebra} \textbf{306} (2006), 378-396.

\bigskip

[25] A. Seress, The minimal base size of primitive solvable permutation
groups, \emph{J. London Math. Soc. }\textbf{53} (1996), 243-255.

\bigskip

[26] B.Sury, Bounded generation does not imply finite presentation,

\emph{Comm. Alg}. \textbf{25 }(1997), 1673-1683

\bigskip

[27] O. I. Tavgen, Bounded generation of Chevalley groups over rings of
algebraic $S$-integers, \emph{Math. USSR Izvestia }\textbf{36 }(1991), 101-128.

\bigskip

[28] B. A. F. Wehrfritz, \emph{Infinite linear groups, }Ergebnisse der Math.
\textbf{76}, Springer-Verlag, Berlin, 1973.

\bigskip

[29] T. R. Wolf, Solvable and nilpotent subgroups of $GL(n,q^{m})$, \emph{Can.
J. Math.} \textbf{34} (1982), 1097-1111.

\bigskip
\begin{verbatim}
L. Pyber
Alfred Renyi Math Institute, Hungarian Academy of Sciences
P. O. Box 127
H-1364 Hungary

pyber@renyi.hu


D. Segal
All Souls College
Oxford OX1 4AL
UK

dan.segal@all-souls.ox.ac.uk
\end{verbatim}

\end{document}